\documentclass{mcom-l}

\newtheorem{thm}{Theorem}[section]
\newtheorem{ap}{Assumption}[section]
\newtheorem{df}[thm]{Definition}
\newtheorem{prop}[thm]{Proposition}
\newtheorem{lm}[thm]{Lemma}
\newtheorem{cor}[thm]{Corollary}

\theoremstyle{remark}
\newtheorem{rk}[thm]{Remark}
\newtheorem{ex}[thm]{Example}

\numberwithin{equation}{section}

\usepackage{graphicx} 
\usepackage{extarrows}
\usepackage{latexsym}
\usepackage{amsmath}
\usepackage{amssymb}
\usepackage{amsfonts}
\usepackage{verbatim}
\usepackage{mathrsfs}
\usepackage[modulo]{lineno} %
\usepackage[latin1]{inputenc}
\usepackage{color}
\usepackage{xcolor}
\usepackage[colorlinks,citecolor=blue,urlcolor=blue]{hyperref}
\usepackage{soul}%
\usepackage{enumitem}
\usepackage{amsthm}

\newcommand{\ee}{\mathbb E}

\newcommand{\pp}{\mathbb P}
\newcommand{\nn}{\mathbb N}
\newcommand{\rr}{\mathbb R}

\newcommand{\yy}{\mathbb Y}

\newcommand{\1}{\mathbf 1}

\newcommand{\DD}{\mathcal D}
\newcommand{\LL}{\mathcal L}

\newcommand{\OO}{\mathcal O}
\newcommand{\PP}{\mathcal P}
\newcommand{\RR}{\mathcal R}

\newcommand{\KK}{\mathcal K}
\newcommand{\WW}{\mathcal W}

\newcommand{\FFF}{\mathscr F}

\newcommand{\PPP}{\mathscr P}

\newcommand{\<}{\langle}
\renewcommand{\>}{\rangle}
\allowdisplaybreaks \allowdisplaybreaks[4]

\newcommand{\dd}{\mathrm{d}}

\newcommand{\abs}[1]{\left\lvert #1 \right\rvert}
\newcommand{\norm}[1]{\left\lVert #1 \right\rVert}

\begin{document}

\title[Non-asymptotic Error Analysis of Explicit MEMs for Non-contractive SODEs]
{Non-asymptotic Error Analysis of Explicit Modified Euler Methods for Superlinear and Non-contractive SODEs}

\author{Zhihui LIU}
\address{Department of Mathematics \& National Center for Applied Mathematics Shenzhen (NCAMS) \& Shenzhen International Center for Mathematics, Southern University of Science and Technology, Shenzhen 518055, China}
% \curraddr{}
\email{liuzh3@sustech.edu.cn}

\author{Xiaojie WANG}
\address{School of Mathematics and Statistics, HNP-LAMA, Central South University, Changsha 410083, P.R. China}
% \curraddr{}
\email{x.j.wang7@csu.edu.cn; x.j.wang7@gmail.com}

\author{Xiaoming WU }
\address{Department of Mathematics, Southern University of Science and Technology, Shenzhen 518055, P.R. China}
% \curraddr{}
\email{12331004@mail.sustech.edu.cn}

 \author{Xiaoyan ZHANG }
\address{School of Mathematics and Statistics, HNP-LAMA, Central South University, Changsha 410083, P.R. China}
% \curraddr{}
\email{232103003@csu.edu.cn}
\thanks{The first and third authors are supported by the National Natural Science Foundation of China, No. 12101296, Guangdong Basic and Applied Basic Research Foundation, No. 2024A1515012348, and Shenzhen Basic Research Special Project (Natural Science Foundation) Basic Research (General Project), Nos. JCYJ20220530112814033 and JCYJ20240813094919026. The second and last authors are supported by the Natural Science Foundation of China, Nos. 12471394, 12071488, 12371417, and Hunan Basic Science Research Center for Mathematical Analysis, No. 2024JC2002.}

%    General info

\date{}

% \dedicatory{This paper is dedicated to our advisors.}


\maketitle

\begin{abstract}
A family of explicit modified Euler methods (MEMs) is constructed for long-time approximations of super-linear SODEs driven by multiplicative noise. The proposed schemes can preserve the same Lyapunov structure as the continuous problems.
Under a non-contractive condition, we establish a non-asymptotic error bound between the law of the numerical approximation and the target distribution in Wasserstein-1 ($\WW_1$) distance through a time-independent weak convergence rate for the proposed schemes. 
As a by-product of this weak error estimate, we obtain an $\OO(\tau|\ln \tau|)$ convergence rate between the exact and numerical invariant measures.
%answering a question left in {\it Z. Liu and X. Wu, arXiv:2411.06049}   
 \end{abstract}

\section{Introduction}
\noindent
% Stochastic differential equations, as a class of important continuous-time models, play a prominent role in various fields ranging from Bayesian inference, statistical physics, machine learning, and scientific computing.  
%

In this paper, we consider the asymptotic behavior and non-asymptotic error analysis for a family of explicit numerical approximations to the  following stochastic ordinary differential equation (SODE):
\begin{align}
\label{SDE} 
    \dd X_t
    =
    b(X_t) 
  \, \dd t
    +
    \sigma(X_t) 
    \, \dd W(t), 
    \quad t >0. 
\end{align}
Here $b: \rr^d \rightarrow \rr^d$ and  $\sigma: \rr^d \rightarrow \rr^{d \times m}$ are measurable functions satisfying certain monotone conditions,
and $W$ is an $\rr^m$-valued Wiener process on a complete filtered probability space  $(\Omega,
\FFF,\{\FFF_t\}_{t \ge 0},\pp)$. 
% Specially, the initial data $x_0: \Omega \rightarrow \rr^d$ is assumed to be $\FFF_0$-measurable.

Under certain conditions, it is known that Eq. \eqref{SDE} admits a unique invariant measure $\mu$.  
To asymptotically sample from $\mu$, one practical way relies on the numerical approximation via the classical Euler--Maruyama (EM) scheme
\begin{align} \label{em}
    Y_{n+1}
    =
    Y_n
    +
    b(Y_n)\tau
    +
    \sigma(Y_n) \delta W_n, 
    \quad Y_0=x_0,
\end{align}
where $\delta W_{n} :=W({t_{n+1}})-W({t_n})$ with $t_n := n \tau$, $n \in \nn$, for a fixed step-size $\tau\in(0,1)$.
Then a non-asymptotic analysis is needed to characterize the error between the target distribution $\mu$ and the law of the EM scheme \eqref{em}.

Non-asymptotic analysis focuses on quantifying the explicit dependency of the above error with respect to the step-size, rather than characterizing asymptotic behavior as $n \to \infty$ or $\tau \to 0$. 
Recently, a series of literature studied non-asymptotic error analysis in an infinite-time regime for EM schemes applied to SODE with coefficients of linear growth across various metrics, including $\WW _p$ distances ($p \geq 1$), total variation (TV) distance, and Kullback--Leibler (KL) divergence: 
For example, \cite{PP23} derived $\OO(\tau_n|\log \tau_n|)$ and $\OO(\tau_n^{1-\varepsilon})$ convergence rates for the EM with varying time-steps $\tau_n$, under $\WW_1$ and TV distances, respectively. For SODEs with multiplicative noise and H\"older drift, \cite{LWX23} obtained error bounds in TV  and $\WW_1$ distance for the EM with varying time-steps. More recently, the authors of \cite{LWW24} established an $\OO(\tau^2)$ KL divergence bound in finite time for the EM applied to SODEs with multiplicative noise under globally Lipschitz conditions. 

For SODEs \eqref{SDE} with super-linear coefficients, it was shown in \cite{HJK11} that the EM scheme \eqref{em} would blow up in $2p$-th moment for all $p \ge 1$. The square function is thus not an appropriate Lyapunov function of the EM method, even though it is a natural Lyapunov function of Eq. \eqref{SDE}.
As the first aim of this paper, we construct a family of explicit modified Euler methods (MEMs, see \eqref{MEM}) that can preserve the same Lyapunov structure, in $2p$-th moment for all $p \ge 1$, of Eq. \eqref{SDE} with super-linear coefficients (see Theorem \ref{thm-Y}).
Such a nice numerical Lyapunov structure plays a vital role in establishing the non-asymptotic error bounds in the infinite-time regime.

In the additive noise case with non-contractive drift coefficients of polynomial growth, \cite{NNZ25} established uniform-in-time convergence rates of order $\OO(\tau)$ and order $\OO(\tau^{1/2})$ in $\WW_1$ and $\WW_2$ distances, respectively,  for a kind of tamed scheme applied to the overdamped Langevin SODE.
Meanwhile, \cite{PWW25} obtained $\OO(\tau|\ln \tau|)$  error estimates under TV distance for projected schemes, 
and \cite{LS25} provided a non-asymptotic convergence analysis in KL divergence for tamed methods, both in the setting of the overdamped Langevin SODE.  

By contrast, the multiplicative noise case with superlinearly growing coefficients remains largely unexplored.
When the superlinear coefficients of multiplicative noise driven SODEs are contractive, i.e., there exist $p^*>1$ and $L >0$ such that  
\begin{align} \label{eq:contractive-condition}
 \langle x - y, b ( x ) -  b ( y ) \rangle +  \tfrac{2 p^*-1}{2} \| \sigma (x) -  \sigma (y) \|    \leq - L | x - y |^2, 
 \quad \forall ~ x, y \in \rr^d,
\end{align}
uniform-in-time strong convergence rates and approximations of invariant measures were examined in \cite{WW24} and \cite{brehier2023approximation, fang2020adaptive, liu2023backward, PWW24}, respectively, for various time-stepping schemes including the backward Euler, tamed, and projected methods. Such a contractive condition offers a contractive property that is key in obtaining uniform-in-time error estimates for the numerical approximations.
Nevertheless, this condition is too restrictive, and in this paper, we attempt to work under a contractivity at infinity condition (see Assumption \ref{A3}), which is much weaker than \eqref{eq:contractive-condition}. More accurately, under  the contractivity at infinity condition, we establish a non-asymptotic error bound between the law of the numerical approximation produced by the constructed MEM \eqref{MEM} and the target distribution in $\WW_1$ distance:
\begin{align}
\label{eq:intro-W1}
    \WW_1
    (\LL(Y_n^{x_0}),\pi)
    & \leq 
    C_* e^{-\lambda \tau n} +
    C^* \tau|\ln \tau|, \quad \forall~ n \in \nn_+,
\end{align} 
see Theorem \ref{thm-W1} for the main result of this paper.
%
%
%This motivates another main aim of our study -- to establish a non-asymptotic error bound for the constructed MEM \eqref{MEM}.
This is achieved by deriving a time-independent weak error estimate of order $\OO(\tau|\ln \tau|)$ under the $\WW_1$ distance for the MEMs (see Theorems \ref{thm-weak}). 
As a by-product of the weak error estimate, we obtain an $\OO(\tau|\ln \tau|)$ convergence rate between the exact and numerical invariant measures (Corollary \ref{cor-inv}), answering a question left in \cite{LW24}. 

It is worth mentioning that our analysis and results are highly non-trivial compared to existing related works mentioned above. Firstly, a new family of numerical methods, including tamed and projected schemes, is introduced to approximate the superlinear SODEs in an infinite-time horizon (see \eqref{ex1} and \eqref{ex2}). Secondly, the possible superlinear growth of the diffusion coefficient $\sigma$ and the non-contractive setting bring essential difficulties in the uniform-in-time analysis. Last but not least, we turn the analysis of the $\WW_1$ error bound into weak error estimates with only Lipschitz continuous test functions $\phi$, different from the standard weak error analysis in \cite{brehier2023approximation, fang2020adaptive, liu2023backward, PWW24}, where $\phi$ is assumed to be smooth enough. 
To overcome difficulties caused by the low regularity of the test functions, we rely on the Bismut--Elworthy--Li formula to explore more delicate regularity estimates of the associated Kolmogorov Equation (see Subsection \ref{subsec:regularity-Kolmogorov-Equation}).

The outline of the paper is as follows. 
Section \ref{sec2}  sets up a framework and presents the main result.
In Section \ref{sec3}, we derive uniform moment estimates for the exact and numerical solutions, along with regularity of the Kolmogorov equation associated with Eq. \eqref{SDE}.
Section \ref{sec4} proves the main result by establishing a time-independent weak error analysis for the proposed schemes. 
%As consequences, we derive non-asymptotic error estimates between the law of the MEMs and the target distribution under $\WW_1$ distance and error estimates between the exact and numerical invariant measures.

\section{Settings and main results}
\label{sec2}

This section will give the assumptions, the proposed schemes, and the main result about the non-asymptotic bounds in $\WW_1$ distance between the laws of the numerical and target distribution.

Let us begin with some preliminaries. 
Throughout this paper, we denote by $\abs{\cdot}$ and $\<\cdot, \cdot\>$ the Euclidean norm and the inner product in $\rr$, $\rr^d$, or $\rr^m$, respectively, if there is no confusion. By $\norm{\cdot}$ we denote the Hilbert--Schmidt norm in $\rr^{d \times m}$. 
In addition, we use $C_b(\rr^d)$ and $\mathrm{\rm Lip}(\rr^d)$ to denote the Banach spaces of all uniformly continuous, bounded mappings and Lipschitz continuous mappings $\phi:\rr^d \rightarrow \rr$ endowed with the seminorms $|\phi|_0=\sup_{x\in\rr^d}|\phi(x)|$ and $|\phi|_{\mathrm{\rm Lip}}=\sup_{x\ne y}|\phi(x)-\phi(y)| |x-y|^{-1}$, respectively.
 Moreover, for $k \in \nn_+$, denote by $C_b^k(\rr^d)$ the subspace of $C_b(\rr^d)$ consisting of all functions with bounded partial derivatives $D^i\phi(x), 1\le i\le k$, endowed with the norm $|\phi|_k:=|\phi|_0+\sum_{i=1}^k \sup_{x\in\rr^d} |D^i\phi(x)|$.

Denote by $\WW_p(\pi_1,\pi_2)$ the $L^p$-Wasserstein distance between two distributions $\pi_1$ and $\pi_2$ on $\rr^d$  \cite[ Section 5.1]{san15}:
\begin{align*}
\WW_p(\pi_1,\pi_2):=\min_{\pi \in \PP(\pi_1,\pi_2)} \Big\{\int_{\rr^d\times \rr^d}|x-y|^p \pi(\dd x, \dd y) \Big\}^{1/p},  
\end{align*}
where $\PP(\pi_1,\pi_2)$ stands for the set of probability measures on $\rr^d\times \rr^d$ with respective marginal laws $\pi_1$ and $\pi_2$.
For convenience, we further denote
$|f|_{L_\omega^p}:=(\int_\Omega f(\omega) \pp(\dd \omega))^{1/p}$
for $f \in L^p(\Omega; \rr^d)$ and $a \vee b:=\max\{a, b\}$ for two positive numbers $a$ and $b$.

We first introduce the following dissipativity condition to establish uniform moment bounds for the SODEs.

\begin{ap}
\label{A1}
There exist positive constants $L_1,L_2, L_3$, $p^{\star} \ge 1$, and $\gamma >1$ such that for any $x,y\in\rr^d$,
     \begin{align}
       &  \< 
         x-y, b(x)-b(y)
         \> 
         +
         \frac{2p^{\star}-1}{2}
         \|\sigma(x)
         -
         \sigma(y)\|^2
         \leq 
         L_1|x-y|^2,\label{mon}\\
   & \< x, b(x)\>
    +
    \frac{p^{\star}(2p^{\star}-1)}{2}
    \|\sigma(x)\|^2
    \leq
   L_2 -L_3|x|^{\gamma+1}.\label{coe}
\end{align}
\end{ap}

Moreover, we require that the coefficients $b$ and $\sigma$ have continuous partial derivatives up to the third order.

\begin{ap}
\label{A2}
   Let $b,\sigma$ have all continuous derivatives up to third order,
    and there exist a positive constant $C$ such that for any $k=1,2,3$ and $x, v_k \in \rr^d$, 
    \begin{align*}
            |
            D^k b(x)
            (v_1,\cdots, v_k)
            | 
            & \leq 
            C
           \big [
            \1_{\gamma \leq k}
            +
            \1_{\gamma >k}
            (1+|x|^{\gamma-k})
           \big ]|v_1| 
            \cdots |v_k|, \\
            |
            D^k \sigma_j(x)
            (v_1, \cdots, v_k)
            |^2 
            &  \leq 
            C
            [
            \1_{\gamma \leq 2k-1}
            +
            \1_{\gamma >2k-1}
            (1+|x|^{\gamma-(2k-1)})
            ]|v_1|^2\cdots|v_k|^2.
    \end{align*}

%%%%%%%%%%%%%
\iffalse
    and
    \begin{align}
       |D^2 b(x)(v_1, v_2)| 
            & \leq 
            C(1+|x|)^{\max \{0, \gamma-2\}} \cdot|v_1| 
            \cdot|v_2|,\\
           |D^2 \sigma_j(x)(v_1, v_2)|^2 
            & \leq  
            C (1+|x|)^{\max \{0, \gamma-3\}} \cdot|v_1|^2 \cdot|v_2|^2,
          |D b(x) v_1| 
            & \leq 
            C(1+|x|)^{\gamma-1}|v_1|,\\  
            |D\sigma_j(x)v_1|^2
            & \leq
            C(1+|x|)^{\gamma-1} |v_1|^2.
    \end{align}
\fi
%%%%%%%%%%%%%%%%%%%%
\end{ap}

Assumption \ref{A2} serves as a kind of polynomial growth condition, and in proofs we will need some implications of this assumption: for any $ x, y,v_1,v_2,v_3 \in \rr^d$,
\begin{align*}
 | D^2 b(x)(v_1, v_2) -D^2 b(y)(v_1, v_2)| 
    \leq 
    C(1+|x|+|y|)^{\max \{0, \gamma-3\}} |x-y| |v_1|  |v_2|, 
\end{align*}
which in turn gives
\begin{align}
\label{Db-v}
       |  D b(x) v_1  -   D b(y) v_1| 
        & \leq 
        C(1+|x|+|y|)
        ^{\max \{0, \gamma-2\}}|x-y| |v_1|, \\
\label{b-grow}
        |b(x)-b(y)| 
        & \leq C(1+|x|+|y|)^{\gamma-1}
        |x-y|.
\end{align}

Following the same idea, Assumption \ref{A2} also ensures, for $j \in\{1, \cdots, m\}$,
\begin{align*}
   & |D^2 \sigma_j(x)(v_1, v_2)
        -
        D^2 \sigma_j(y)(v_1, v_2)
        |^2
        \leq  
        C (1+|x|+|y|)^{\max \{0, \gamma-5\}}  |x-y|^2|v_1|^2|v_2|^2.
\end{align*}
This in turns yields 
\begin{align}
    |D \sigma_j(x) v_1-D \sigma_j(y) v_1|^2 
        & \leq 
        C 
        (1+|x|+|y|)
        ^{\max \{0, \gamma-3\}}|x-y|^2  |v_1|^2, \label{Ds-v}\\
 |\sigma_j(x)-\sigma_j(y)
        |^2 
        & \leq C(1+|x|+|y|)^{\gamma-1}
        |x-y|^2. \label{s-grow}
\end{align}

We introduce a contractivity at infinity condition, instead of the global contractive condition \eqref{eq:contractive-condition}, which is used to establish the exponential contraction for Eq. \eqref{SDE}.

\begin{ap}
\label{A3}
 $\lambda_0^{-2} I \geq \sigma\sigma^T \geq \lambda_0^2 I$ for some constant $\lambda_0 \in(0,1)$, and
there exist positive constants $K_1, K_2, r_0$ such that $\sigma_0:=\sqrt{\sigma\sigma^T-\lambda_0^2 I}$ satisfies
\begin{align*}
        & 
\| \sigma_0(x)-\sigma_0(y)\|
        ^2
        -
        \frac{\< \sigma(x)-\sigma(y),x-y\>^2}
        {|x-y|^2}
        +
        \< b(x)-b(y), x-y\> \nonumber\\
        & \leq  
     \{(K_1+K_2) 1_{\{|x-y| \leq r_0\}}-K_2\}|x-y|^2, 
        \quad \forall~ x, y \in \rr^d.
\end{align*}
%%%%%%%%%%%%%%%%%%%%
%
%
\iffalse
then there exist constants $c, \lambda>0$ such that
\begin{align}
\label{ex-decay}
    \WW_1
    \left(\mu_1 P_t, \mu_2 P_t\right) 
    \leq 
    c e^{-\lambda t}
    \WW_1(\mu_1, \mu_2), 
    \quad t \geq 0, 
    \,\mu_1, \mu_2 \in \PPP\left(\rr^d\right).
\end{align}
\fi
%
%
%
\end{ap}

%
%
%
\iffalse
\begin{prop}
    (Exponential ergodicity in $\WW_1$-distance)
    Let Assumption \ref{A3} hold. 
    There exist some constant $c, \lambda>0$, independent of $t$, such that the semi-group $P_t$ and its invariant measure $\pi$ satisfy
    \begin{align}
    \label{ex-decay}
    \WW_1
   (\mu P_t, \pi) 
    \leq 
    C_{\star} e^{-\lambda t}
    \WW_1(\mu, \pi), 
    \quad t \geq 0, 
    %\,\mu \in \PPP(\rr^d).
\end{align}
where $\mu:= \LL(x_0)$ is the initial distribution of the Eq. \eqref{SDE}.
\end{prop}
Thanks to the definition of $\WW_1$ distance \cite[Definition 6.1]{san15} and the following Theorem \ref{thm-X}, we can get 
\begin{align}
\label{ex-decay-x}
    \WW_1
    \left(
  \LLaw(X^X_t), \pi\right)
    & \leq
    C_{\star} e^{-\lambda t}
    \WW_1
    \left(
  \LLaw(x_0), \pi\right)\nonumber\\
    & =
    C_{\star} e^{-\lambda t}
    \inf\{
   \ee
    [|x_0-y|], ~~
  \LLaw(y)=\pi\}\nonumber\\
    & \leq
    C_{\star} e^{-\lambda t}
    \inf 
    \big\{
   \ee
    [|x_0|^2]^{1/2}
    +
   \ee
    [|y|^2]^{1/2},~~
  \LLaw(y)=\pi
    \big\}\nonumber\\
    & \leq 
    C_{\star} e^{-\lambda t}
    \left(
   \ee
    [1+|x_0|^{2}]
    \right)^{1/2}.
\end{align}
\fi
%
%
%

\begin{rk} \label{rk-erg} 
Under Assumption \ref{A3}, \cite{wang20} derived the $\WW_1$-exponential ergodicity of Eq. \eqref{SDE}: for any initial distribution $\mu$ on $\rr^d$ of Eq. \eqref{SDE}, there exist positive constants $C_*$ and $\lambda$ such that the corresponding Markov semigroup $(P_t)_{t \geq 0}$ and its invariant measure $\pi$ satisfy  
    \begin{align}
    \label{ex-decay}
    \WW_1
    (\mu P_t, \pi) 
    \leq 
    C_* e^{-\lambda t}
    \WW_1(\mu, \pi), 
    \quad \forall ~t \geq 0.
\end{align}  
\end{rk}

\iffalse
\begin{rk}
\label{rk-mon}
     It is noteworthy that Assumption \ref{A3} implies a  monotonicity condition of the drift and diffusion
     coefficients of SDEs \eqref{SDE} as follows: there 
     exists a positive constant $L > 0$, for any $p^{\star} \in [1,\infty)$, such that
     \begin{align}
     \label{mon}
         \< 
         x-y, b(x)-b(y)
         \> 
         +
         {\color{red}
         \frac{2p^{\star}-1}{2}}
         \|\sigma(x)
         -
         \sigma(y)\|_F^2
         \leq 
         L|x-y|^2, 
         \quad \forall x, y \in \rr^d.
     \end{align}
\end{rk}
\fi

To numerically solve the  superlinear SODE \eqref{SDE}, we propose the following MEMs:
\begin{align}
\label{MEM} 
Y_{n+1}
=
\PPP(Y_n)
+
b_\tau(\PPP(Y_n)) \tau
+
\sum_{j=1}^m
\sigma_{j,\tau}(\PPP(Y_n))  
\delta W_{n},
\end{align}
where $\PPP:\rr^d \rightarrow \rr^d$ is a kind of modification function, and $b_\tau(x)$ and $\sigma_{j,\tau}(x)$, $j \in \{1,2,\cdots,m\}$, are $\rr^d$-valued measurable functions.
In addition, we need some conditions to set up a general framework for the MEM \eqref{MEM}.

\begin{ap}
    \label{A4}
There exist constants $ \alpha_1 \ge 1$ and $\alpha_2 \ge 2$ such that for any $x \in \rr^d$ and $j \in \{1,2,\cdots,m\}$,
    \begin{align}
            | b_\tau(\PPP(x)) | & \leq C|b(x)|,\\
            | \sigma_{j,\tau}(\PPP(x))  | &\leq C|\sigma_j(x)|,\\
            |b_\tau(\PPP(x))-b(\PPP(x))|+
            |\sigma_{j,\tau}(\PPP(x))-\sigma_j(\PPP(x))|
            &\leq C \tau (1+|x|^{\alpha_1}), \\
    |\PPP(x)| &\leq  |x|,   \label{MEM-rel} \\
    |\PPP(x)-x| & \leq  C \tau^2 |x|^{\alpha_2}.
\end{align}
\end{ap}

%%%%%%%%%%%%%%%%
\iffalse
\begin{ap}
\label{A4}
$(1)$ There exist some constant $ \alpha_1, \alpha_2\in [1,\infty)$ such that, for $j \in \{1,2,\cdots, m\}$ and $ \forall x \in \rr^d$,
    \begin{align}
            |
            b_\tau(\PPP(x))
            |
            \vee
            |
            b_\tau(\PPP(x))
            |
            & \leq
            C|f(x)|,\\
            |\sigma_{j,\tau}(x)|
            & \leq
            C|\sigma_j(x)|,\\
            |b_\tau(x)-f(x)|
            & \leq 
             C h (1+|x|^{\alpha_1}),\\
             |\sigma_{j,\tau}(x)-\sigma_j(x)|
             & \leq 
             C h (1+|x|^{\alpha_2}).
    \end{align}
$(2)$ For any $ x \in \rr^d$,
\begin{align}
\label{MEM-rel}
    |\PPP(x)|
    \leq 
    |x|.
\end{align}
$(3)$ There exists a constant $\alpha_3 \in [1,\infty)$ such that, for any $ x \in \rr^d$,
\begin{align}
    |\PPP(x)-x|
    \leq 
    C h^2 |x|^{\alpha_3}.
\end{align}
\end{ap}
\fi
%%%%%%%%%%%%%%%%%%%

\begin{ap}
\label{A5}
There exist positive constants $K_3,K_4,K_5,K_6, p^*\ge 1$ and $\alpha_3<2$ such that for all $x\in \rr^d$,
\begin{align}
2\<\PPP(x),b_\tau(\PPP(x))\>+\binom{2p^{\star}}{2}\sum_{j=1}^m
|\sigma_{j,\tau}(\PPP(x)) |^2+\tau|b_\tau(\PPP(x))|^2
& \le K_3-K_4|\PPP(x)|^2, \label{mon*}\\ 
\sqrt{\tau} \sum_{j=1}^m|
\sigma_{j,\tau}(\PPP(x)) |^2 
& \le K_5 +K_6|\PPP(x)|^{\alpha_3}. \label{s-tau}
\end{align}
\end{ap}

We mention that an analogous framework was established in \cite{pang2024antithetic} for a modified Milstein method without Levy area that was used to approximate SODEs over a finite time interval.
Also, it should be noted that the above framework includes the following tamed and projected schemes as special cases.

\begin{ex}
\begin{enumerate}
\item
\textbf{Tamed Euler method (TEM):}  
    \begin{align} \label{ex1}
\begin{split}
        & \PPP(x)
        :=
        x, \quad 
        b_\tau( \PPP(x))
        :=
        \frac{b(x)}{(1+\tau|x|^{4(\gamma-1)})^{1/4}}, \\
        & \sigma_{j,\tau}( \PPP(x))
        :=
        \frac{\sigma_j(x)}{(1+\tau|x|^{4(\gamma-1)})^{1/4}},
        \quad x \in \rr^d.
        \end{split}
    \end{align}
It is not hard to verify that the above TEM  satisfies Assumptions \ref{A4} and \ref{A5} (with $\alpha_1=5\gamma-4, \alpha_2=2$ and $\alpha_3=3-\gamma$); the proof can be seen in \cite{LW24}.

 \item
\textbf{Projected Euler method (PEM):}
\begin{align}\label{ex2}
\begin{split}
         \PPP(x)
         &:=
         \left\{\begin{array}{ll}
         \min \{1, \tau^{-\frac{1}{2\gamma}}|x|^{-1}\} x, & x \neq \textbf{0}, \\
         0, &  x=\textbf{0},
         \end{array} 
         \right.\\
         b_\tau(\PPP(x))
         &:=
         b(\PPP(x)), 
         \quad
         \sigma_{j,\tau}(\PPP(x))
         :=
         \sigma_j(\PPP(x)), \quad x \in \rr^d.
         \end{split}
\end{align}
%with $\gamma $ being given in Assumption \ref{A2}. 
%Also, we let $ \textbf{0}/0 =\textbf{0}$ for $\textbf{0}:=(0,\cdots,0)^T \in \rr^d$.
%For the particular case $\gamma=1, \PPP=I$ is the identity operator. 
Evidently, Assumptions \ref{A4} and  \ref{A5} can be fulfilled with $\alpha_1=1, \alpha_2=4\gamma+1$, and  $\alpha_3=1$. We mention that such a projected scheme was first proposed in \cite{BIK16} to strongly approximate superlinear SODEs in finite time and later used by \cite{PWW24} to approximate invariant measures of superlinear SODEs.
\end{enumerate}
\end{ex}

Denote by $\{Y_n^{x_0}\}_{n \geq 0}$ and $\{\LL(Y_n^{x_0})\}_{n \geq 0}$ the solution and its law, respectively, of the MEM \eqref{MEM} with the initial value $Y_0^{x_0}=x_0$.
Our main result of this paper is formulated as follows, whose proof is postponed to Section \ref{sec4}.

\begin{thm}
\label{thm-W1}
Let Assumptions \ref{A1}-\ref{A5} hold and let $\pi$ be the unique invariant measure of the Markov semigroup $(P_t)_{t \geq 0}$ of Eq. \eqref{SDE}. 
There exist positive constants $C_*$, $\lambda$, and $C^*$ such that for any $\tau \in (0, 1)$ and $n \in \nn$, 
\begin{align}
\label{W1}
    \WW_1
    (\LL(Y_n^{x_0}),\pi)
    & \leq 
    C_* e^{-\lambda \tau n} +
    C^* \tau|\ln \tau|.
\end{align}
\end{thm}
% We obtain the following results as a direct consequence of the above theorem.

As a by-product of the uniform weak error estimates in Theorem \ref{thm-weak}, we have the following time-independent  $\OO(\tau|\ln \tau|)$ convergence rate between the numerical invariant measure of the MEM \eqref{MEM} and the exact invariant measure of Eq. \eqref{SDE}, answering a question left in \cite{LW24}. 

\begin{cor}\label{cor-inv}
Let Assumptions \ref{A1}-\ref{A5} hold and $\pi_\tau$ be the invariant measure of the MEM \eqref{MEM}.
For any $\tau \in (0, \tau_{\max})$ with $\tau_{\max}:=\min\{1/K_4,1\}$, there exists a constant $C$ such that 
\begin{align*} 
|\pi(\phi)-\pi_\tau(\phi)| \le C \tau\ln|\tau|, \quad \forall ~ \phi \in {\rm Lip}(\rr^d).
\end{align*}
\qed
\end{cor}

\section{Estimates of MEMs and Kolmogorov Equation}
\label{sec3}

 In this section, we aim to derive uniform moment estimates for the exact solution of  Eq. \eqref{SDE} and the MEM \eqref{MEM}, along with uniform regularity of the associated Kolmogorov equation.

\subsection{Moment Estimates}

We begin with the following known moment estimates of the solution $\{X_t^{x_0}\}_{t \geq 0}$ of Eq. \eqref{SDE} with the initial value $X_0^{x_0}=x_0$; see e.g. \cite{WW24}.

\begin{thm}
\label{thm-X}
    Let Assumption \ref{A1} hold.
   For  any $x_0 \in L_w^{2p}$ and  $p \in [1,p^{\star}]$, there exists a positive constant $C$, independent of $t \geq 0$, such that
   \begin{align} 
       \ee|X^{x_0}_t|^{2p}
        & \leq  C
       (1+\ee|x_0|^{2p}), \quad t \ge 0,  \label{X-est} \\
        \ee|X_t^{x_0}-X_s^{x_0}|^{2p}
        & \leq  C (1+\ee |x_0|^{2p})
         (t-s)^p, \quad 0 \leq t -s  \leq 1. 
         \label{X-hold}
     \end{align} 
\qed
\end{thm}

Next, we give the uniform moment bounds of the MEM \eqref{MEM}.

\begin{thm}
\label{thm-Y}
    Let Assumptions \ref{A4} and \ref{A5}  hold.
    Then for any $p \in [1,p^{\star}] \cap \nn_+$  and any $\tau\in(0,\tau_{\max})$, there exist a constant $C$, independent of $n, \tau$, such that 
       \begin{align}
   \label{Y-est}
       \ee[|Y_{n+1}|^{2p}|\FFF_{t_n}]&\le (1-K_4\tau/4)|Y_n |^{2p}+C\tau.
\end{align} 
\end{thm}

\begin{proof}
For $n \in \nn_+$ and $\tau \in (0, \tau_{\max})$ with $\tau_{\max}=\min\{1/K_4,1\}$, taking square on both sides of \eqref{MEM} gives
\begin{align*}
|Y_{n+1}|^2
&=
|\PPP(Y_n)
+
b_\tau(\PPP(Y_n)) \tau
+
\sum_{j=1}^m
\sigma_{j,\tau}(\PPP(Y_n))  
\delta W_n|^2 \\
&=|\PPP(Y_n)|^2+2\tau\<\PPP(Y_n),b_\tau(\PPP(Y_n))\>+\sum_{j=1}^m|
\sigma_{j,\tau}(\PPP(Y_n))  
\delta W_n|^2 \\
&\quad +\tau^2|b_\tau(\PPP(Y_n))|^2
+2 \sum_{j=1}^m\<\PPP(Y_n)
+ b_\tau(\PPP(Y_n)) \tau,
\sigma_{j,\tau} (\PPP(Y_n))  \delta W_n \>.
\end{align*}
Taking the conditional expectation $\ee[\cdot|\FFF_{t_n}]$ on both sides of the above  equation and using It\^o isometry, the conditions \eqref{mon*}, \eqref{MEM-rel}, and the equality
\begin{align*}
\ee\Big[2 \sum_{j=1}^m \<\PPP(Y_n)
+ b_\tau(\PPP(Y_n)) \tau, 
\sigma_{j,\tau} (\PPP(Y_n))  \delta W_n \> \Big]=0,
\end{align*}
followed from the facts that $\PPP(Y_n)$ is $\FFF_{t_n}$-measurable and that $\delta W_n$ is independent of $\FFF_{t_n}$, we obtain
\begin{align*}
& \quad \ee[|Y_{n+1}|^2|\FFF_{t_n}] \\
&=|\PPP(Y_n)|^2 + \tau [2\<\PPP(Y_n),b_\tau(\PPP(Y_n))\>+\sum_{j=1}^m|
\sigma_{j,\tau}(\PPP(Y_n))|^2  +
\tau|b_\tau(\PPP(Y_n))|^2] \\
& \le (1-K_4\tau)|\PPP(Y_n)|^2+K_3\tau.
\end{align*}
This completes the proof of \eqref{Y-est} when $p=1$.

For any integer $p \ge 2$, taking $2p$-th power on both sides of  \eqref{MEM}, we deduce
\begin{align*}
&\quad|Y_{n+1}|^{2p}\\
&=
|\PPP(Y_n)
+
b_\tau(\PPP(Y_n)) \tau
+
\sum_{j=1}^m
\sigma_{j,\tau}(\PPP(Y_n))  
\delta W_n|^{2p} \\
&= |\PPP(Y_n)
+
b_\tau(\PPP(Y_n))  \tau|^{2p}  \\
&\quad+
\sum_{k=2}^{2p}
\binom{2p}{k}
|\PPP(Y_n)
+
b_\tau(\PPP(Y_n))  \tau|^{2p-k}(\sum_{j=1}^m
\sigma_{j,\tau}(\PPP(Y_n))  
\delta W_n)^k \\
&\quad
+ 2p|\PPP(Y_n)
+
b_\tau(\PPP(Y_n))  \tau|^{2p-2}
\Big\<\PPP(Y_n)
+
b_\tau(\PPP(Y_n))  \tau,\sum_{j=1}^m
\sigma_{j,\tau}(\PPP(Y_n))  
\delta W_n\Big\>.
\end{align*}
Then taking the conditional expectation  $\ee[\cdot |\FFF_{t_n}]$
 on both sides of the above equation and using It\^o isometry, we obtain
 \begin{align*}
\ee[|Y_{n+1}|^{2p}|\FFF_{t_n}]
&= |\PPP(Y_n)+b_\tau(\PPP(Y_n))  \tau|^{2p}\\
&\quad+p(2p-1) \tau|\PPP(Y_n)+b_\tau(\PPP(Y_n))  \tau|^{2p-2} \sum_{j=1}^m|
\sigma_{j,\tau}(\PPP(Y_n)) |^2\\
&\quad+
\sum_{k=2}^{p}
 \tau^k
\binom{2p}{2k}
|\PPP(Y_n)
+
b_\tau(\PPP(Y_n))  \tau|^{2p-2k}\sum_{j=1}^m|
\sigma_{j,\tau}(\PPP(Y_n))  
|^{2k}.
\end{align*}
For any $ \tau \in (0, \tau_{\max})$, from the condition \eqref{mon*} we have
\begin{align}
\label{y+f+g}
|\PPP(Y_n)
+
b_\tau(\PPP(Y_n)) \tau|^2+p(2p-1)\tau\sum_{j=1}^m|
\sigma_{j,\tau}(\PPP(Y_n))  
|^2
\le
(1-K_4\tau)|\PPP(Y_n)|^2+K_3\tau.
\end{align}
It follows that 
 \begin{align}
 \label{y+b}
 |\PPP(Y_n)
+
b_\tau(\PPP(Y_n))  \tau|^{2p}
&\le \sum_{k=0}^{p}
\binom{p}{k} 
[(1-K_4 \tau)|\PPP(Y_n)|^2]^{p-k}(K_3 \tau)^k \nonumber\\
& \le
(1-K_4 \tau/2)|\PPP(Y_n)|^{2p}-E_1(|\PPP(Y_n)|),
 \end{align}
where  we denote
\begin{align*}
E_1(\xi):=
\frac{1}{2}K_4\tau|\xi|^{2p}-\sum_{k=1}^{p}
\binom{p}{k} \tau^k (K_3)^k|\xi|^{2p-2k},
\quad \xi \in \rr_+.
\end{align*}
Set
\begin{align*}
    H_1(p):= 2p
\begin{pmatrix}
    p\\
     \lceil p/2 \rceil
\end{pmatrix}
(1+K_3 )^p (\min\{1,K_4\})^{-1},
\quad  p \in \nn_+.
\end{align*}
When $\xi >H_1(p)$, for $1\le k \le p$, $E_1(\xi)$ can be bounded below by
\begin{align*}
E_1(\xi) \ge \tau\sum_{k=1}^{p}[\frac{1}{2p}K_4|\xi|^{2p}-
\binom{p}{ \lceil p/2 \rceil}  (1+K_3)^p|\xi|^{2p-2k}
]\ge 0.
\end{align*}
%
%
%
\iffalse
we have
\begin{align*}
H_1(p)\ge \max_{1\le k \le p} \max\Big\{\Big[ 
\tfrac{2p
\begin{pmatrix}
    p\\
     \lceil p/2 \rceil
\end{pmatrix}
(1+K_3)^p}{K_4}\Big]^{\frac{1}{2k}}
\Big\}.
\end{align*}
Thus, when $|\PPP(Y_n) |> H_1(p)$, we have 
\begin{align*}
& |\PPP(Y_n)
+
b_\tau(\PPP(Y_n))  \tau|^{2p-2}[|\PPP(Y_n)
+b_\tau(\PPP(Y_n))  \tau|^2 + \tau| \sum_{j=1}^m|
\sigma_{j,\tau}(\PPP(Y_n))  
|^2]\\
&\le (1-K_4\tau)|\PPP(Y_n) |^{2p}+K_3\tau|\PPP(Y_n) |^{2p-2}.
\end{align*}
\fi
%
%
%
By  \eqref{s-tau} and Young inequality, we obtain from $|\PPP(Y_n) |> H_1(p)$ that 
\begin{align*}
&\ee[|Y_{n+1}|^{2p}|\FFF_{t_n}]\\
&\le (1-K_4\tau)|\PPP(Y_n) |^{2p}+K_3\tau|\PPP(Y_n) |^{2p-2} \\
&\quad
+\sum_{k=2}^{p}
\binom{2p}{2k}
\tau|\PPP(Y_n) |^{2p-2k}( K_5 +K_6|\PPP(Y_n)|^{\alpha_3})^k \\
&\le (1-\frac{1}{2}K_4\tau)|\PPP(Y_n) |^{2p}+ \sum_{k=2}^{p}
\binom{2p}{2k}
\tau2^{k-1}K_5^k|\PPP(Y_n) |^{2p-2k} \\
&\quad+
\sum_{k=2}^{p}
\binom{2p}{2k}
\tau2^{k-1}K_6^k|\PPP(Y_n) |^{2p-2k+\alpha_3k} +
\frac{(2p-2)^{p-1}K_3^p}{K_4^{p-1}p^p}\tau\\
&\le (1-K_4\tau/4)|\PPP(Y_n) |^{2p}+\frac{(2p-2)^{p-1}K_3^p}{K_4^{p-1}p^p}\tau-E_2(|\PPP(Y_n) |)-E_3(|\PPP(Y_n) |),
\end{align*}
 where for any $\xi \in \rr_+$,
 \begin{align*}
E_2(\xi)&:=
\frac{1}{8}K_4\tau|\xi|^{2p}-\sum_{k=2}^{p}
\binom{2p}{2k} \tau^k (2K_5)^k|\xi|^{2p-2k}, \\
E_3(\xi)&:=
\frac{1}{8}K_4\tau|\xi|^{2p}-\sum_{k=2}^{p}
\binom{2p}{2k} \tau^k (2K_6)^k|\xi|^{2p-2k+\alpha_3k}.
\end{align*}
It is not hard to show when $\xi \ge H_2(p):=8(p-1)\binom{2p}{p}(1+2K_5)^p (\min\{1,K_4\})^{-1}$ and 
$\xi \ge H_3(p):=8(p-1)\binom{2p}{p}(1+2K_6)^p (\min\{1,K_4\})^{-1}$, we have $E_2(\xi)\ge 0$ and $E_3(\xi)\ge 0$.

\iffalse
 For any $\xi \in \rr_+$,
$E_2(\xi)$ can be bounded below by
\begin{align*}
E_2(\xi)\ge \tau\sum_{k=2}^{p}[
\frac{1}{8(p-1)}K_4|\xi|^{2p}-
\binom{2p}{p}  (1+K_5)^p|\xi|^{2p-2k}].
\end{align*}
Observe that, when $\xi \ge H_2(p):=\frac{8(p-1)\binom{2p}{p}(1+K_5)^p}{\min\{1,K_4\}}$,
we have $E_2(\xi)\ge 0$.
For $E_3(\xi)$, we have
\begin{align*}
E_3(\xi)\ge \tau\sum_{k=2}^{p}[
\frac{1}{8(p-1)}K_4|\xi|^{2p}-
\binom{2p}{p}  (1+K_6)^p|\xi|^{2p-2k+\alpha_3k}].
\end{align*}
Similarly, when $\xi \ge H_3(p):=\frac{8(p-1)\binom{2p}{p}(1+K_6)^p}{\min\{1,K_4\}}$,
we obtain $E_3(\xi)\ge 0$.
\fi

Denote $H(p):=\max\{H_1(p),H_2(p),H_3(p)\}$.
By the above estimates, we obtain from $|\PPP(Y_n)|\ge H(p)$ that
\begin{align*}
\ee[|Y_{n+1}|^{2p}|\FFF_{t_n}]\le
(1-K_4\tau/4)|\PPP(Y_n) |^{2p}+\frac{(2p-2)^{p-1}K_3^p}{K_4^{p-1}p^p}\tau.
\end{align*}
For $|\PPP(Y_n)| < H(p)$, we have
\begin{align*}
\ee[|Y_{n+1}|^{2p}|\FFF_{t_n}]\le
(1-K_4\tau/4)|\PPP(Y_n) |^{2p}+c(p)\tau,
\end{align*}
where 
\begin{align*}
c(p)&:=(|H(p)|^2+K_3\tau)
\sum_{k=1}^{p}\binom{p}{k}K_3^k|H(p)|^{2p-2k}+\frac{(2p-2)^{p-1}K_3^p}{K_4^{p-1}p^p}\tau \\
&\quad
+\sum_{k=2}^{p}
\binom{2p}{2k}
2^{k-1}K_5^k|H(p) |^{2p-2k} +
\sum_{k=2}^{p}
\binom{2p}{2k}
2^{k-1}K_6^k|H(p) |^{2p-2k+\alpha_3k}.
\end{align*}
This completes the proof.
\iffalse
Thus, by the above estimates and \eqref{MEM-rel}, we have
\begin{align*}
\ee[|Y_{n+1}|^{2p}|\FFF_n]&\le (1-K_4\tau/4)|Y_n |^{2p}+c(p)\tau \\
& \le (1-K_4\tau/4)^{n+1}|x_0 |^{2p}+\frac{4c(p)}{K_4}.
\end{align*}
\fi
\end{proof}

\subsection{Regularity of Kolmogorov Equation}
\label{subsec:regularity-Kolmogorov-Equation}

To carry out the weak error analysis in terms of 
$|\ee \varphi(X_t)-\ee \varphi(Y_N)|$ for all $\phi \in {\rm Lip}(\rr^d)$, we rely on the Kolmogorov equation associated with Eq. \eqref{SDE}.
Under the previous assumptions, 
the function $u(\cdot, \cdot):[0, \infty) \times \rr^d \rightarrow \rr$ defined by
\begin{align}
\label{def-u}
    u(t, x)
    :=
   \ee \phi(X_t^x)
\end{align}
is the unique classical solution of the Kolmogorov equation \eqref{Komgv}: 
\begin{align}
\label{Komgv}
     \left\{
     \begin{array}{l}
     \partial_t u(t, x)
     =
     D u(t, x) b(x)
     +
     \frac{1}{2}
     \sum_{j=1}^m 
     D^2 u(t, x)
    (\sigma_j(x), \sigma_j(x)), 
     \\
     u(0, x) = \phi(x).
     \end{array}
     \right.
\end{align}

In the following, we derive more delicate estimates of $u$ defined in \eqref{def-u} and its derivatives, which are vital for the subsequent error 
analysis.
% The existence and the uniqueness of the mean-square derivatives up to the third order can be derived owing to \eqref{mon} (see \cite[Proposition 1.3.5]{cer01}).  
% The following lemma establishes a priori estimates for the mean-square derivative of the exact solution of Eq. \eqref{SDE}.
%
We begin with the following estimates on the directional (mean-square) derivative $\DD X_t^\cdot$.

\begin{lm}
\label{lm-DX}
Let Assumptions \ref{A1} and \ref{A2} hold. 
For any $q\in[1, p^{\star}], ~ t \ge 0$, and $ x, v_1,v_2,v_3 \in \rr^d$, there exists a constant $C$ such that 
\begin{align}
        | \DD X_t^x v_1 |_{L_\omega^{2q}}
       & \leq 
        e^{C t}|v_1|, \label{DX-est}\\
        |\DD^2 X_t^x(v_1, v_2) |_{L_\omega^{2q}}
        & \leq 
        e^{Ct} 
        [
        \1_{1 < \gamma \le 2}
        +
        \1_{\gamma >2}
       (
        1+
        |x|^{\gamma-2}
        )
        ]
        |v_1|
        |v_2|,\label{D2X-est} \\
\label{D3X-est}
  |\DD^3 X_t^x(v_1, v_2,v_3)|
    _{L_\omega^2}
    & \leq 
     e^{Ct} 
    [
    \1_{1 < \gamma \le 2}
    +
    \1_{\gamma >2}
   (
    1+
    |x|^{\gamma-2}
    )
    ]
    |v_1|
 |v_2|
  |v_3|.
\end{align}
\end{lm}

\begin{proof} 

In what follows, without loss of generality, we denote
\begin{align*}
    \eta^{v_1}(t, x)
    :=
    \DD X_t^{x} v_1, 
    \quad 
    \xi^{v_1, v_2}(t, x)
    :=
    \DD^2 X_t^{x}(v_1, v_2), 
\end{align*}
\begin{align*}
    \zeta^{v_1, v_2,v_3}(t, x)
    :=
    \DD^3 X_t^{x}(v_1, v_2,v_3), 
    \quad x, v_1, v_2, v_3\in \rr^d.
\end{align*}

For any $x, y \in \rr^d $, it follows from \eqref{mon} that
\begin{align}
\label{D-coup}
    \< Db(x) y, y\> +
    \frac{2p^{\star}-1}{2}
    \|D\sigma(x)y\|^2\leq L_1|y|^2.
\end{align}
For $\eta^{v_1}(t,x)$, we have
\begin{align*}
\left\{\begin{array}{l}
   \dd \eta^{v_1}(t,x)
    =
    D b(X_t^x) \eta^{v_1}(t, x) 
   \dd t
    +
    \sum_{j=1}^m
    D \sigma_j(X_t^x) \eta^{v_1}(t, x) 
   \dd  W_{j,t}, \\
    \eta^{v_1}(0,x)=v_1.
\end{array}\right.
\end{align*}
Define a stopping time as
$$
\widetilde{\tau}_n^{(1)}
=
\inf
\{
s \geq 0:|\eta^{v_1}(s,x)|>n 
~\text { or }~ 
|X_s^x|>n
\}.
$$
Using the It\^o formula and the Cauchy-Schwarz inequality shows, for $q \in [1,p^{\star}]$,
\begin{align}
\label{1var-Ito}
        &
       \ee
        | \eta^{v_1}(t \wedge \widetilde{\tau}_n^{(1)}, x)|^{2q}\nonumber\\
        & \leq
|v_1|^{2q}
        +
        2q 
       \ee
        \int_0^{t \wedge \widetilde{\tau}_n^{(1)}}
 |\eta^{v_1}( x)|^{2(q-1)} [\< \eta^{v_1}(x), D  b(X^x) \eta^{v_1}(x) \>\nonumber\\
        & \qquad
         + \frac{2q-1}{2} \sum_{j=1}^m
        |D \sigma_j(X^x) \eta^{v_1}(x)  |^2
        ]
       \dd  s\nonumber\\
        & \leq
    |v_1|^{2q}
        +
      C
       \ee \int_0^{t \wedge \widetilde{\tau}_n^{(1)}}
        |\eta^{v_1}(x)|^{2q}
       \dd s.
\end{align}
We omit the integration variable here and after when there is an integration to lighten the notation. 
Using Gronwall inequality and Fatou lemma, we get \eqref{DX-est}.

Similarly, $ \xi^{v_1, v_2}(t, x)$ satisfies $ \xi^{v_1, v_2}(0, x)=0 $ and
\begin{align}
  \dd  \xi^{v_1, v_2}(t, x)
    &= 
   (
        Db(X_t^x) 
        \xi^{v_1, v_2}(t, x)
        +
        D^2 b(X_t^x)
        (
            \eta^{v_1}(t, x), 
            \eta^{v_2}(t, x)
       )
  )
   \dd  t \nonumber \\
    &\quad +
    \sum_{j=1}^m
    (
        D \sigma_j(X_t^x) \xi^{v_1, v_2}(t, x)  +
        D^2 \sigma_j(X_t^x)
        (\eta^{v_1}(t, x), 
        \eta^{v_2}(t, x))
    ) 
   \dd  W_{j, t}.
   \end{align}

Following the same idea as \eqref{1var-Ito} (and introducing an analogous stopping time if necessary), we obtain that, for some $q\in [1, p^{\star})$,
\begin{align}
\label{2var-Ito}
        &
       \ee |\xi^{v_1, v_2} (t, x)|^{2q} \nonumber\\
        & \leq
         2q \ee\int_0^{t } | \xi^{v_1, v_2}(x)|^{2(q-1)}\< \xi^{v_1, v_2}(x),D^2b(X^x)  (\eta^{v_1}(x), \eta^{v_2}(x)) \>\dd  s\nonumber\\
       &\quad+
        2q \ee\int_0^{t } | \xi^{v_1, v_2}(x)|^{2(q-1)}[\< \xi^{v_1, v_2}(x),Db(X^x)   \xi^{v_1, v_2}(x) \>\dd  s\nonumber\\
        & \qquad + \frac{(2q-1)(1+\epsilon)}{2} \sum_{j=1}^m |D \sigma_j(X^x) \xi^{v_1, v_2}(x)|^2
        ]\dd  s \nonumber\\
              & \quad +
      C\ee\int_0^{t }|
        \xi^{v_1, v_2}
        (x)
        |^{2q}\dd s +
      C\sum_{j=1}^m\ee\int_0^{t }|
        D^2 \sigma_j(X^x) 
        \eta^{v_1}(x)
                \eta^{v_2}(x))
        |^{2q}
       \dd  s \nonumber\\
       &\le C\ee\int_0^{t }|\xi^{v_1, v_2}(x)|^{2q}\dd s 
       + C\ee\int_0^{t }|D^2b(X^x)  (\eta^{v_1}(x), \eta^{v_2}(x))|^{2q} \dd s \nonumber\\
       &\quad+ C\sum_{j=1}^m\ee\int_0^{t }|
        D^2 \sigma_j(X^x) 
        \eta^{v_1}(x)
                \eta^{v_2}(x))
        |^{2q} \dd s\nonumber\\
        &=:C\ee\int_0^{t }|\xi^{v_1, v_2}(x)|^{2q}\dd s +I_{11}+I_{12}.
\end{align}
It remains to estimate $I_{11}$ and $I_{12}$.
In the case $1< \gamma \leq 2$,  we employ Assumption \ref{A2},  H\"older inequality, and \eqref{DX-est} to arrive at
\begin{align*}
      I_{11} \leq 
         Ce^{Ct}|v_1|^{2q}|v_2|^{2q}.
\end{align*}
In the case $\gamma >2$, we get
\begin{align*}
I_{11} & \leq 
        C\int_0^{t } |(1+|X^x| )^{\gamma-2} |_{L_\omega^{2\rho_3 q}}^{2q}
        | \eta^{v_1}(x)|_{L_\omega^{2\rho_4 q}} ^{2q}
        |\eta^{v_2}(x)|_{L_\omega^{2\rho_5 q}}^{2q} \dd s\\
        & \leq 
        Ce^{Ct}(1+|x|^{2q(\gamma-2)})|v_1|^{2q}| v_2|^{2q},
\end{align*}
for some positive constants $\rho_j$ satisfying $\sum_{j=1}^31/\rho_j=1$.
In the same manner,
\begin{align*}
       I_{12} & \leq         
       Ce^{Ct} [ \1_{\gamma \in(1,3]}+\1_{\gamma >3}(1+|x|^{q(\gamma-3)})]|v_1|^{2q} |v_2|^{2q}.
\end{align*}
Combining the above two estimates and \eqref{2var-Ito} and using Gronwall inequality, we obtain \eqref{D2X-est}.

For the last term $\zeta^{v_1,v_2,v_3}(t,{x})$, we get
\begin{align*}
        &
       \dd   
        \,\zeta^{v_1, v_2, v_3}(t, x)\\
        & =  
        (
        D b(X_t^x) \zeta^{v_1, v_2, v_3}(t, x)
        +
        D^2 b(X_t^x)
       (
        \eta^{v_1}(t, x), 
        \xi^{v_2, v_3}(t, x)
        )\\
        & \quad +
        D^2b(X_t^x)
       (
        \xi^{v_1, v_3}(t, x), \eta^{v_2}(t, x)
       ) 
        +
        D^2 b(X_t^x)
        (
        \xi^{v_1, v_2}(t, x), \eta^{v_3}(t, x)
        )\\
        & \quad +
        D^3 b(X_t^x)
        (
        \eta^{v_1}(t, x), 
        \eta^{v_2}(t, x), 
        \eta^{v_3}(t, x)
        )
        ) 
       \dd  t \\
        & \quad +
        \sum_{j=1}^m
        (
        D \sigma_j
      (X_t^x) 
        \zeta^{v_1, v_2, v_3}(t, x)
        +
        D^2 \sigma_j(X_t^x)
        (
        \eta^{v_1}(t, x), 
        \xi^{v_2, v_3}(t, x)
        )\\
        & \quad + 
         D^2 \sigma_j(X_t^x)
        (
         \xi^{v_1, v_3}(t, x), \eta^{v_2}(t, x)
         )
         +
         D^2 \sigma_j
         (X_t^x)
         (
         \xi^{v_2, v_3}(t, x), \eta^{v_1}(t, x)
         )\\
         & \quad +
         D^3 \sigma_j
         (X_t^x)
         (
         \eta^{v_1}(t, x), 
         \eta^{v_2}(t, x), 
         \eta^{v_3}(t, x)
         )
         ) 
        \dd  W_{j, t}
         \\
         & = :
         (
         Db
         (X_t^x) \zeta^{v_1, v_2, v_3}(t, x)
         +
         H(X_t^x))\dd  t\\
         & \quad +
         \sum_{j=1}^m
         (
         D \sigma_j(X_t^x) \zeta^{v_1, v_2, v_3}(t, x)
         +
         G_j(X_t^x)
         ) 
        \dd  W_{j, t},
\end{align*}
where
        $\zeta^{v_1, v_2, v_3}(0, x)=0$.
%              
 %%%%%%%%%%%%%%%%%%%%%%%%%%%%%%
        \iffalse
i.e.,
\begin{align*}
\left\{
\begin{aligned}
\mathrm{d}\zeta^{v_1, v_2, v_3}(t, x) &= 
(
Db(X_t^x) \zeta^{v_1, v_2, v_3}(t, x)
+ H(X_t^x)
)\dd t \\
&\quad + \sum_{j=1}^m 
(
D\sigma_j(X_t^x) \zeta^{v_1, v_2, v_3}(t, x)
+ G_j(X_t^x)
)\dd W_{j,t}, \\
\zeta^{v_1, v_2, v_3}(0, x) &= 0.
\end{aligned}
\right.
\end{align*}
\fi
%
Then we have
\begin{align}
\label{3var-Ito}
        &\quad \ee|\zeta^{v_1, v_2, v_3}(t , x)|^2\nonumber\\
        & \le
        C \ee   \int_0^{t }|\zeta^{v_1, v_2, v_3}( x)|^2 \dd s
        + C \ee   \int_0^{t }|H(X^x)|^2 \dd s
    +C\sum_{j=1}^m \ee \int_0^t|G_j(X^x)|^2\dd  s\nonumber\\
        &=: C \ee   \int_0^{t }|\zeta^{v_1, v_2, v_3}( x)|^2 \dd s+I_{21}+I_{22}.
        \end{align}
We begin with the estimation of the term $I_{21}$ as follows:
\begin{align*}
        I_{21} & \leq
    C    \int_0^{t }  |
        D^2 b(X^x)
        (
        \eta^{v_1}(x), 
        \xi^{v_2, v_3}(x)
        )
      |_{L_\omega^2}^2 \dd s
      \\
        & \quad +
    C    \int_0^{t } |
        D^2 b(X^x)
        (
        \xi^{v_1, v_3}(x), \eta^{v_2}(x)
        )
      |_{L_\omega^2}^2 \dd s
        \\
        & \quad +C    \int_0^{t }
      |
        D^2 
        b(X^x)
        (\xi^{v_1, v_2}(x), \eta^{v_3}(x)
        )
        |_{L_\omega^2}^2 \dd s
        \\
        & \quad +C    \int_0^{t }
       |
        D^3 b(X^x)
        (
        \eta^{v_1}(x), 
        \eta^{v_2}(x), 
        \eta^{v_3}(x)
        )
        |_{L_\omega^2}^2 \dd s  \\
        &=:  I_{211}+ I_{212}+ I_{213}+ I_{214}.
\end{align*}
We mention that the analysis of $ I_{211}, I_{212}, I_{213}$ is similar.
Taking $ I_{211}$ as example, in the case $1 <\gamma \leq 2$, we employ  H\"older inequality, \eqref{DX-est}, and \eqref{D2X-est} to get, 
\begin{align*} 
    I_{211}  \leq
        C
        e^{Ct}
        |v_1|
         |v_2|
         |v_3|.
\end{align*}
In the case $\gamma>2$, we obtain
\begin{align*}
     I_{211} \leq 
        C
        e^{Ct}
        (1+|x|^{\gamma-2})
        |v_1| |v_2| |v_3|.
\end{align*}
For $I_{214}$, we have
\begin{align*}
       I_{214}
         \leq 
        C
        e^{Ct}
        [
        \1_{\gamma \in(1,3]}
        +
        \1_{\gamma >3}
        (1+|x|^{\gamma-3})
        ]
        |v_1|  |v_2| |v_3|.
\end{align*}
%
%%%%%%%%%%%%%%%%%%%%
%
\iffalse
\begin{align}
        & 
        |
        D^3 f(X_s^x)
        (
        \eta^{v_1}(x), 
        \eta^{v_2}(x), 
        \eta^{v_3}(x)
        )
        |_{L^2(\Omega, \rr^d)} \\
        & \leq 
        C
        |
        (
        1+|X^X_t|
        )^{\gamma-3} 
        |\eta^{v_1}(x)|
        \cdot
        |\eta^{v_2}(x)|
        \cdot
        |\eta^{v_3}(x)|
        |_{L^2(\Omega, \rr)}\\
        & \leq 
        C
        |
        (
        1+|X^X_t|
        )^{\gamma-3} |_{L^{2\rho^{'}_1}(\Omega, \rr)}
        |\eta^{v_1}(x)|_{L^{2\rho^{'}_2}(\Omega, \rr^d)}
        \cdot
        |\eta^{v_2}(x)|_{L^{2\rho^{'}_3}(\Omega, \rr^d)}
        \cdot
        |\eta^{v_3}(x)|_{L^{2\rho^{'}_4}(\Omega, \rr^d)}\\
        & \leq
        C
        e^{3Ls}
        |v_1|
        \cdot
        |v_2|
        \cdot
        |v_3|.
\end{align}
\fi
%%%%%%%%%%%%%%%%%%%%%
%
%
That is to say, $I_{214}$  can be  controlled by $I_{211}$. 
Hence, we get
\begin{align*}
      I_{21}
        & \leq
        C
        e^{Ct}
        [
        \1_{1 < \gamma \le 2}
        +
        \1_{\gamma >2}
        (1+|x|^{2\gamma-4})
        ]
        |v_1|^2
        |v_2|^2
        |v_3|^2.
\end{align*}
Similarly, it is not hard to obtain
\begin{align*}
       I_{22}& \leq
        C
        e^{Ct}
        [
        \1_{\gamma \in(1,3]}
        +
        \1_{\gamma >3}
        (1+|x|^{\gamma-3})
        ] |v_1|^2|v_2|^2|v_3|^2.
\end{align*}
Plugging these estimates into \eqref{3var-Ito} yields
\begin{align*}
       \ee
        [
        |\zeta^{v_1, v_2, v_3}
        (t, x)|^2
        ] 
        & \leq 
      C
       \ee
        \int_0^{t }
        |\zeta^{v_1, v_2, v_3}(x)|^{2}
       \dd  s\nonumber\\
        & \quad +
        Ce^{Ct}
        [
        \1_{1 < \gamma \le 2}
        +
        \1_{\gamma >2}
        (1+|x|^{2\gamma-4})
        ]
        |v_1|^2
        |v_2|^2
        |v_3|^2.
\end{align*}
By Gronwall inequality, we conclude \eqref{D3X-est}. 
\end{proof}

In general, if $\phi \in C_b^k(\rr^d)$ and the solution $X^x_t$ is $k$-times differentiable, then $u(x,t)$ is $k$-times differentiable by the chain rule and an explicit formula for the derivatives of $u(x,t)$ that can be given in terms of derivatives of $\phi$ and $X^x_t$. 
However, instead of its derivatives, we want to express the derivative of $u$ in terms of $\phi$ only.
For this purpose, we recall the Bismut--Elworthy--Li (BEL) formula (see \cite{cer01, EKL94, PP23}). 
% Such a formula applies when $X^x_t$ is mean-square differentiable and the diffusion term is non-degenerate.

\begin{lm}
\label{lm-BEL} 
Let $\{X_t^x\}_{t \geq 0}$ be differentiable and $\sigma$ be invertible with inverse $\sigma^{-1}$. 
For any $\varphi \in C_b(\rr^d)$, $t >0$, 
and $x, v \in \rr^d$,
\begin{align} \label{bel}
    \< 
    D(P_t \varphi(x)), v
    \>
    =
    \frac{1}{t} 
   \ee
    \Big[
    \int_0^t
   \<
    \sigma^{-1}(X^x)
    \DD X^x v, \dd W
   \> 
    \varphi(X_t^x)
    \Big].
\end{align} 
Moreover, \eqref{bel} remains true if $\varphi$ is a Borel function with polynomial growth.
\qed
\end{lm}

Recalling the fact that $P_t \phi(x)=u(t, x)$ by \eqref{def-u}, the BEL formula \eqref{bel} paves the way to presenting an expression for the derivatives of $u(t, \cdot)$, where it only requires $\phi \in \mathrm{\rm Lip}(\rr^d)$.
Combining Lemma \ref{lm-DX} with the BEL formula \eqref{bel}, we are now in the position to obtain some estimates about the derivatives of $u(t, \cdot), t >0$.

\begin{lm}
\label{lm-Du}
   Let Assumptions \ref{A1}-\ref{A3} hold. 
   For any $t>0$ and $\phi \in {\rm Lip}(\rr^d)$, $u(t, \cdot) \in C_b^3(\rr^d)$ and satisfy the following estimates for any $x,v_1,v_2,v_3 \in \rr^d$.
   \begin{enumerate}
   \item[(1)]
  For $t\in(0,1]$, there exists a positive constant $C$ such that
    \begin{align} 
        |D u(t, x) v_1| & \leq 
        C (1+|x|)|v_1|, \label{Du-est} \\
      |D^2 u(t, x)(v_1, v_2)| 
        & \leq 
        C t^{-1/2}
        [1+
%        (
        \1_{1 < \gamma \le 2}
        |x|
        +
        \1_{\gamma >2}
        |x|^{\gamma-1}
%        )
        ]
        |v_1| |v_2|,  \label{D2u-est} \\  
        |D^3 u(t, x)(v_1, v_2,v_3)| 
        & \leq  C t^{-1}
        [
        1+
        \1_{1 < \gamma \le 2}|x|
        +
        \1_{\gamma >2}
        |x|^{\gamma-1}
        ]
        |v_1||v_2||v_3|. \label{D3u-est}
    \end{align}
    \item[(2)]
    For $t \in(1, \infty)$, there exist positive constants $C_{\star}$ and $\lambda$ such that
    \begin{align}
    \label{Du-est+}
        |D u(t, x) v_1| & \leq 
        C_{\star} 
        e^{-\lambda(t-1)}
        (1+|x|)
        |v_1|, \\
    \label{D2u-est+}
        |
        D^2 u(t, x)
        (v_1, v_2)
        | 
        & \leq 
        C_{\star} 
        e^{-\lambda(t-1)}
        [1+
%        (
        \1_{1 < \gamma \le 2}
        |x|
        +
        \1_{\gamma >2}
        |x|^{\gamma-1}
%        )
        ]
        |v_1||v_2|, \\
    \label{D3u-est+}
        |
        D^3 u(t, x)
        (v_1, v_2, v_3)
        | 
        & \leq 
        C_{\star} 
        e^{-\lambda(t-1)}
        [
        1+
        \1_{1 < \gamma \le 2}|x|
        +
        \1_{\gamma >2}
        |x|^{\gamma-1}
        ]
        |v_1||v_2||v_3|.
    \end{align}
    \end{enumerate}
\end{lm}

\begin{proof}
For some function $\phi\in \mathrm{\rm Lip}(\rr^d)$, there exists a constant $M$ such that
\begin{align}
\label{tf-lip}
    |\phi(x)-\phi(y)|
    \leq
    M|x-y|,
    \quad \forall x, y \in \rr^d.
\end{align}
Then we can calculate the first,  the second, and the third derivative of
$u(t, x)
    :=
   \ee \phi(X_t^x) $
with respect to $x$, by Lemma \ref{lm-BEL}. 
%Recalling proof of Lemma \ref{lm-DX}, we use the same notations $\eta^{v_1}(t,x)$, $\xi^{v_1,v_2}(t,x)$ and $\zeta^{v_1,v_2,v_3}(t,x)$ to denote the first, the second and the third derivatives of $X^X_t$ with respect to $x$ for $t > 0$, respectively. 
%
The proof is divided into two cases.
%based on the time parameter $t>0$.

(1)
  When $t \in(0,1]$, we have
\begin{align*}
        D u(t, x) v_1
       & =
        \frac{1}{t} 
       \ee
        \Big[
        \int_0^t
     \<
        \sigma^{-1}(X^x)
        \eta^{v_1}(x), 
       \dd  W
      \> 
        \phi(X_t^x)
        \Big]
        \\
        & =
        \frac{1}{t} 
       \ee
        \Big[
        \int_0^t
     \<
        \sigma^{-1}(X^x)
        \eta^{v_1}(x), 
       \dd  W
      \> (
        \phi(X_t^x)
        -
         \phi(x))
        \Big].
\end{align*}
By Theorem \ref{thm-X}, Lemma \ref{lm-DX}, Assumption \ref{A3}, H\"older inequality and It\^o isometry, one can easily get
\begin{align*}
        |D u(t, x) v_1|
        & \leq 
        \frac{
        |\phi(X_t^x)
        -
        \phi(x)|_{L_\omega^2}}
        {t}
        \Big |
        \int_0^t
       \<
        \sigma^{-1}(X^x)
        \eta^{v_1}(x), 
       \dd  W
       \>
        \Big |_{L_\omega^2}
        \\
        & \leq
        \frac{
        M
        |X_t^x-x|_{L_\omega^2}}
        {t}
        \Big (
        \int_0^t
       |
        \sigma^{-1}(X^x)
        \eta^{v_1}(x)
        |_{L_\omega^2}^2\dd  s
        \Big )^{1 / 2} 
        \\
        & \leq
        C
        \frac{1}{\sqrt{t}}
        (e^{C t}-1)^{1 / 2}
        (1+|x|)
        |v_1|.
\end{align*}
In view of the Markov property of Eq. \eqref{SDE}, we have
\begin{align*}
        u(t, x)
        & =
        P_t f(x)
        =
        P_{t-s}P_s f(x),
\end{align*}
which, by  setting $s=\frac{t}{2}$, leads to the following expression with respect to $D u(t, \cdot)$:
\begin{align*}
    D u(t, x) v_1
    =
    \frac{2}{t} 
   \ee
    \int_0^{\frac{t}{2}}
  \<
    \sigma^{-1}(X^x)
    \eta^{v_1}(x), 
   \dd  W
 \> 
    u(t/2, X_{t/2}^x).
\end{align*}
Since both the functions $u(t, \cdot)$ and $X_t^\cdot $ are continuously differentiable,  the function $u(t / 2, X_{t / 2}^\cdot)$ is also continuously differentiable. 
Applying the chain rule, we obtain the first derivative of $u(t / 2, X_{t / 2}^x)$ as
\begin{align*}
    D
    (
    u(t / 2, X_{t / 2}^x)
    )
    v_2
    =
    D u
    (
    t / 2, X_{t / 2}^x
 ) 
    \eta^{v_2}(t / 2, x),
\end{align*}
which directly implies a formula for the second derivative of $u(t, \cdot)$ as follows:
\begin{align*}
        &
        D^2 u(t, x)(v_1,v_2)\\
        & = \frac{2}{ t} 
       \ee
        \int_0^{\frac{t}{2}}
     \<
        \sigma^{-1}(X^x)
        \eta^{v_1}(x), 
       \dd  W
      \> 
        D u(t/2, X_{t / 2}^x) 
        \eta^{v_2}(t/2,x)\\
        & \quad +\frac{2}{t} 
       \ee
        \int_0^{\frac{t}{2}}
       \<
        \sigma^{-1}(X^x)
        \xi^{v_1, v_2}(x),\dd  W
     \> 
        (
        u(t/2, X_{t / 2}^x)
        -
        u(0,x)
       )\\
        & \quad 
        -
        \frac{2}{t}
       \ee
        \int_0^{\frac{t}{2}}
        \<
        \sigma^{-2}(X^x)
        D \sigma(X^x)
        \eta^{v_2}(x)
        \eta^{v_1}(x), 
       \dd  W
        \> 
        (
        u(t/2, X_{t/2}^x)
        -
        u(0,x)
        )\\
        &=:I_1+I_2+I_3.
\end{align*}
For  $I_1$, using Lemma \ref{lm-DX}, H\"older inequality and the It\^o isometry shows
\begin{align*}
        |I_1| 
         & \le
        \frac{2}{t}
        | 
        D u(t / 2, X_{t / 2}^x) 
        \eta^{v_2}(t / 2, x)
     |_{L_\omega^2}
        (
        \int_0^{\frac{t}{2}}
        |
        \sigma^{-1}(X^x)
        \eta^{v_1}(x)
        |^2_{L_\omega^2} 
       \dd  s
        )^{1 / 2}\\
        & \leq
        \frac{C}{t^{3/2}}
        (e^{ C t}-1)
        (1+|X_{t / 2}^x|_{L_\omega^4})
        |\eta^{v_2}(t / 2, x)|_{L_\omega^4}
        |v_1|\\
        & \leq 
        \frac{C}{t^{3/2}}
        (e^{ C t}-1)
        e^{Ct/2}
        (1+|x|)
        |v_1|
        |v_2|.
\end{align*}
For  $I_2$, we need a bound on 
        $
        |
        u(t/2, X_{t/2}^x)
        -
        u(0,x)
      |
        _{L_\omega^2}$.
According to Theorem \ref{thm-X}, the Jensen inequality and \eqref{tf-lip}, we obtain
\begin{align*}
        |
        u
        (
        t / 2, X_{t / 2}^x
        )
        -
        u(0,x)
        |^2
        _{L_\omega^2}
              & =
       \ee
        |
       \ee
       [
        \phi(X^X_t)
        -
        \phi(x)
        \mid
        X^x_{t/2}
        ]
        |^2
       \\
        & \leq
       \ee
       [
       \ee
      [
        |
        \phi(X^X_t)
        -
        \phi(x)
        |^2
        \mid
        X^x_{t/2}
       ]
        ]\\
        & \leq
        C (1+|x|^2)t.
\end{align*}
In the same manner, we have
\begin{align*}
        |I_2| 
        & \leq
        \frac{C}
        {\sqrt{t}} 
        (
        e^{Ct/2}-1
        )^{1/2}
        [
        1+
        \1_{1 < \gamma \le 2}
        |x|
        +
        \1_{\gamma >2}
        |x|^{\gamma-1}
        ]
        |v_1|
        |v_2|,
        \end{align*}
        and
        \begin{align*}
         |I_3|
        & \leq
        \frac{C}
        {\sqrt{t}}
        (
        e^{Ct}-1
        )^{1/2}
        (1+|x|^{(\gamma-1)/2})
        |v_1|
        |v_2|.
\end{align*}
Taking these estimates into account and using the inequality $e^{C t}-1 \leq (e^C-1)t$ for any $t\in[0,1]$ and $C>0$, we get \eqref{Du-est} and \eqref{D2u-est}.

(2) When $t \in(1, \infty)$,
using Kantorovich's dual formula and Assumption \ref{A3} yields
\begin{align}
\label{W1-contrac}
    \WW_1(
  \LL(X^x_t),\pi)
    :=
    \sup_{|\phi|_{\rm Lip} \leq 1}
    \Big|
   \ee \phi(X_t^x)
    -
    \int_{\rr^d}
    \phi(x)
  \pi( \dd x)
    \Big|
    \leq
    C_{\star} 
    e^{-\lambda t}
    (1+|x|).
%    (d^{1 / 2}+|x|), \quad \forall x \in \rr^d.
\end{align}
Recalling \eqref{def-u}, when $t>1$, the Markov property immediately implies that
\begin{align*}
    u(t, x)
    =
   \ee u(t-1, X_1^x).
\end{align*}
Further, from \eqref{W1-contrac} we arrive at
\begin{align*}
    |
    u(t-1, x)
    -
    \int_{\rr^d} \phi(x) \pi(\dd  x)
    | 
    \leq 
    C_{\star}
    e^{-\lambda (t-1)}
    (1+|x|).
%    (d^{1 / 2}+|x|), \quad \forall x \in \rr^d.
\end{align*}
Inspired by \cite{bre14}, here we choose
\begin{align*}
    \Phi_t(x)
    :=
    u(t-1, x)
    -
    \int_{\rr^d} 
    \phi(x) 
    \pi(\dd  x),
\end{align*}
with
\begin{align*}
    |\Phi_t(x)|
    \leq
    \KK (\Phi_t, x)
    =
    C_{\star}
    e^{-\lambda (t-1)}
    (1+|x|).
%    (d^{1 / 2}+|x|).
\end{align*}
Now, for some function $\bar{\phi}:\rr^d \rightarrow \rr$, there exists $\KK(\bar{\phi},x) \geq 0$, which depends on $\bar{\phi}$ and $x$, such that,
$$
|\bar{\phi}(x)|\leq \KK(\bar{\phi},x), 
\quad \forall x \in \rr^d.
$$
We define 
$V(t, x)
:=
\ee \bar{\phi}(X_t^x)$, 
and its first and the second derivative with respect to $x$ can be founded in \cite[Theorem 4.5]{PWW25}.
Then, by the Markov property, we infer
\begin{align*}
       \ee
        [
        \Phi_t(X_1^x)
        ] 
        & =
       \ee
        [
        u(t-1, X_1^x)
        ]
        -
        \int_{\rr^d}
        \phi(x) 
        \pi(\dd  x)
        =
        u(t, x)
        -
        \int_{\rr^d} 
        \phi(x) \pi(\dd  x),
\end{align*}
which leads to
\begin{align*}
    u(t, x)
    =
   \ee
  [
    \Phi_t(X_1^x)
  ]
    +
    \int_{\rr^d} \phi(x) \pi(\dd  x)
    =
    V(1, x)
    +
    \int_{\rr^d} \phi(x) \pi(\dd  x),
\end{align*}
%Since $\widetilde{\Phi}: \rr^d \rightarrow \rr$ 
%%belongs to $C(\rr^d)$, there exists a constant $\KK (\Phi, x)>0$ depending on $\Phi$ and $x$, such that
%\begin{align*}
%    |\widetilde{\Phi}(x)| 
%    \leq 
%    \KK (\Phi, x), 
%    \quad \forall ~x \in \rr^d.
%\end{align*}
With the estimates of $DV(t,x)v_1$ and $D^2V(t,x)(v_1,v_2)$ at $t=1$ in \cite[Theorem 4.5]{PWW25}, we obtain \eqref{Du-est+} and \eqref{D2u-est+}.
The estimate of $D^3 u(t,x)(v_1,v_2,v_3)$ is proved similarly and will not be repeated here for brevity.
\end{proof}

To overcome the discontinuity of the MEM \eqref{MEM}, we introduce its continuous interpolation $\{\yy^n(t)\}_{t \in[t_n, t_{n+1}]}$ and $ n \in\{0,1, \cdots, N-1\}$,
\begin{align}
\label{yy}
    \left\{\begin{array}{l}
   \yy^n(s)
    =
   \yy^n(t_n)
    +
    b_\tau
    (
   \yy^n(t_n)
  )(s-t_n)
    +
    \sum_{j=1}^m
    \sigma_{j,\tau}
    (
   \yy^n(t_n)
   )(W_s-W_{t_n}), 
   \\
   \yy^n(t_n)
     =
    \PPP(Y_n) .
\end{array}\right.
\end{align}
It is straightforward to verify that $\yy^n(t_{n+1})=Y_{n+1}$ and the continuous-time extensions $\{\yy^n(t)\}_{t \in[t_n, t_{n+1}]}$, $ n \in\{0,1, \cdots, N-1\}$ are not continuous on the grid points.
Nevertheless, by Theorem \ref{Y-est} and H\"older inequality, we have the following estimates of the process $\yy^n$.

\begin{lm}
\label{lm-yy}
    Let  $p \in \nn_+,j\in \{1,2,\cdots,m\}$, and Assumptions \ref{A2}, \ref{A4}, and \ref{A5} hold. 
    For any $t_n\le s \le t \le t_{n+1} $ with $n \in \nn_+$, the following estimates hold:
    \begin{align*}
       \ee
        |\yy^n  (s)|^{2 p}
             &   \leq C
        (
        1+
       \ee
     |x_0|^{2 p}
        ), \quad \forall~x_0 \in L_w^{2p}, \\
           \ee
            |\yy^n  (s)
            -
           \yy^n
            (t_n)|^{2 p}
           & \leq 
            C (t-s)^p
            (
            1+
           \ee
            |x_0|^{2p\gamma}
                        ),  \quad \forall~x_0 \in L_w^{2p\gamma},  \\
           \ee
            |
            b(\yy^n  (s))
            -
            b(\yy^n  (t_n))
            |^{2 p}
            & \leq 
            C (t-s)^p
            (
            1+
           \ee
            |x_0|^{2p(2\gamma-1)}
           ),\quad \forall~x_0 \in L_w^{2p(2\gamma-1)},  \\
           \ee
            |
            \sigma_j(\yy^n  (s))
            -
            \sigma_j(\yy^n  (t_n))
            |^{2 p}
                        & \leq 
            C (t-s)^p
            (
            1+
           \ee
            |x_0|^{p(3\gamma-1) }
           ),  \quad \forall~x_0 \in L_w^{p(3\gamma-1)}.
    \end{align*}
    \qed
\end{lm}

\section{Weak Convergence Analysis and Proof of Theorem \ref{thm-W1}}
\label{sec4}

Equipped with  uniform moment estimates of Eq. \eqref{SDE} and the MEM \eqref{MEM}, along with uniform regularity of the associated Kolmogorov equation,
this section aims to prove Theorem \ref{thm-weak} and  bound the $\WW_1$ distance between the law of the MEM \eqref{MEM} and the target distribution induced by Eq. \eqref{SDE} as follows:
  \begin{align*}
        \WW_1
        (
      \LL(Y^{x_0}_N),\pi
        )
        & \leq
        \WW_1
        (
      \LL(X^{x_0}_{N\tau}),
        \pi
        )
        +
        \WW_1
        (
      \LL(Y^{x_0}_N),
      \LL(X^{x_0}_{N\tau})
        )
        \nonumber\\
        & =
         \WW_1(\LL(X^{x_0}_{N\tau}),\pi)+
        \sup_{|\phi|_{\rm Lip}\leq1}|  \ee \phi(Y^{x_0}_N)-\ee \phi(X^{x_0}_{N\tau}) |.
\end{align*}

For $x_0 \in L^2(\Omega; \rr^d)$, by \eqref{ex-decay}, Theorem \ref{thm-X}, and the definition of $\WW_1$-distance, we get 
\begin{align}
\label{ex-decay-x}
    \WW_1
    (
  \LL(X^{x_0}_t), \pi)
     & \leq
    C_{\star} e^{-\lambda t}
    \WW_1
    (
  \LL(x_0), \pi)  
%    & = C_{\star} e^{-\lambda t} \inf\{ \ee [|x_0-y|], ~~  \LL(y)=\pi\}  \nonumber \\
%    & \leq C_{\star} e^{-\lambda t}  \inf  \big\{ \ee  [|x_0|^2]^{1/2}  +  \ee  [|y|^2]^{1/2},~~ \LL(y)=\pi \big\}    
%    \nonumber \\
% & \leq  C_{\star} e^{-\lambda t} \WW_2 (\LL(x_0), \pi)
     \leq 
    C_{\star} e^{-\lambda t}
    ( \ee
    [1+|x_0|^{2}]
    )^{1/2}, \quad t \ge 0.
\end{align}

To facilitate the subsequent analysis, we define the following terms:
\begin{align*}
         J_1&=:
        \sum_{n=0}^{N-1} 
       \ee
        [
        u(T-t_n,\yy^n_n)
       ]
        -
       \ee
       [
        u(T-t_n, Y_n)
        ],\\
        J_{21}&=:
        \sum_{n=0}^{N-1} 
       \ee
        \int_{t_n}^{t_{n+1}}
        D u(T-s,\yy^n)
       (
        b_\tau(\yy^n_n)
        -
        b(\yy^n  )
       ) 
       \dd  s,\\
        J_{22}
       & =:
        \frac{1}{2}
        \sum_{n=0}^{N-1} 
        \sum_{j=1}^m
       \ee
        \int_{t_n}^{t_{n+1}}
        D^2 u
        (
        T-s,\yy^n
       )
        (
        \sigma_{j,\tau}
       (
       \yy^n_n
        ), 
        \sigma_{j,\tau}
        (
       \yy^n_n
        )
        )  \\
        &
        \qquad
        \qquad  -
        D^2 u
       (
        T-s,\yy^n
     )
        (
        \sigma_{j}
        (
       \yy^n
      ), 
        \sigma_{j}
      (
       \yy^n
       )
        )
       \dd s.
\end{align*}
In the following, we provide the relevant estimates.

\begin{lm}
\label{lm-J1}
Let Assumptions \ref{A1}-\ref{A5} hold. 
%Let $\left\{\yy^n (s)\right\}_{s \in\left[t_n, t_{n+1}\right]}$ and $\{Y_n\}_{n \in \{0,1,\cdots,N-1\}}$ be defined as \eqref{yy} and \eqref{MEM}, $n \in\{0,1, \cdots, N-1\}, N \in \nn$. 
For any $x_0 \in L_w^{2\alpha_2}$, there exists a positive constant $C$ such that
\begin{align}
\label{J1-est}
        |J_1|
        & \leq
        C
        (
        1+
        |x_0|^{\alpha_2+1}
        _{L_\omega^{2\alpha_2}
        }
        ) 
        \tau.
\end{align}

\iffalse
\begin{align}
\label{R3-est}
|R_3|
        & \leq
        C
        [
        1+
        \1_{1 < \gamma \le 2}
        |x_0|
        _{L_\omega^{\max\{4\alpha_1,\alpha_1+\gamma+2\}}}
        ^{\max\{2\alpha_1+1, (\gamma+2\alpha_1+3)/2\}}\nonumber\\
        & \qquad +
        \1_{\gamma >2}
        |x_0|_{L_\omega^{\max\{2 \gamma+\alpha_1,2\alpha_1+\gamma-1\}}}^{\max\{\gamma+2\alpha_1-1,(3 \gamma+2\alpha_1-1)/2\}}
        ] \tau,
\end{align}
where $\alpha_1,\alpha_2, \gamma $ comes from Assumption \ref{A2} and \ref{A4}.
\fi

\end{lm}
\begin{proof}
Set $N_1 \in \nn$ such that $N_1 \tau\leq 1< (N_1+1) \tau$. Then using Taylor expansion and Lemma \ref{lm-Du} yields
\begin{align*}
        |J_1| 
        & \leq  
        \sum_{n=0}^{N-N_1-1}
       \ee
        \int_0^1 
        |
        D u(T-t_n, v_1)
        (\yy^n_n-Y_n) 
        |
       \dd  \bar{r} 
       \\
        & \quad +
        \sum_{n=N-N_1}^{N-1}
       \ee 
        \int_0^1 
        |
        D u(T-t_n, v_1)(\yy^n_n-Y_n) 
        |
       \dd  \bar{r}\\
       &=:J_{11}+J_{12},
\end{align*}
where $v_1(\bar{r}):=Y_n+\bar{r}(\yy^n  _n-Y_n), \bar{r} \in[0,1]$.
By Lemma \ref{lm-Du},   Theorem \ref{Y-est}, and  H\"older inequality, the estimates of $J_{11}$ and $J_{12}$ can be bounded as
\begin{align*}
        J_{11} 
        & \leq  
        C_{\star}
        \sum_{n=0}^{N-N_1-1} 
        e^{-\lambda (T-t_n-1)}
       \ee
        \int_0^1
      (
        1+|v_1)|
        )
        |\yy^n_n
        -Y_n|
       \dd  \bar{r}
    \\
        & \leq
        C_{\star}
        \sum_{n=0}^{N-N_1-1} 
        e^{-\lambda (T-t_n-1)}
        (
        1+
        \sup_{r \ge 0}
        |Y_r|
        _{L_\omega^2}
        )
      |
       \yy^n_n-Y_n
       |
        _{L_\omega^2}
     \\
        & \leq
        C_{\star} \tau^2
        \sum_{n=0}^{N-N_1-1} 
        e^{-\lambda (T-t_n-1)}
        \sup_{r \ge 0}
        (
    (
        1+
        |Y_r|
        _{L_\omega^2}
       )
        |Y_r|^{\alpha_2}
        _{L_\omega^{2\alpha_2}}
        )\\
        & \leq
        C (1+
        |x_0|^{\alpha_2+1}
        _{L_\omega^{2\alpha_2}}
        ) 
        \tau,
\end{align*}
and 
\begin{align*}
        J_{12} 
                 \leq
        C
        (1+
        |x_0|^{\alpha_2+1}
        _{L_\omega^{2\alpha_2}}
        ) 
        \tau,
\end{align*}
where it is straightforward to obtain that
\begin{align*}
    \sum_{n=0}^{N-N_1-1}
    e^{-\lambda(T-t_n-1)} \tau \quad \text { and }
    \quad \sum_{n=N-N_1}^{N-1} \tau
\end{align*}
are uniformly bounded with respect to $T$.
The desired assertion \eqref{J1-est} follows. 
\iffalse
\begin{align*}
    R_1 \leq 
    C
    (
    1+
    |x_0|^{\alpha_2+1}
    _{L_\omega^{2\alpha_2}
  }
    ) 
    \tau.
\end{align*}
\fi 
\end{proof}

\begin{lm}
\label{lm-J21}
Let Assumptions \ref{A1}-\ref{A5} hold. 
%Let $\left\{\yy^n (s)\right\}_{s \in\left[t_n, t_{n+1}\right]}$ and $\{Y_n\}_{n \in \{0,1,\cdots,N-1\}}$ be defined as \eqref{yy} and \eqref{MEM}, $n \in\{0,1, \cdots, N-1\}, N \in \nn$. 
For any $x_0 \in L_w^{(4\gamma) \vee (6\gamma-4)}$,
there exists a positive constant $C$ such that
\begin{align}
\label{J21-est}
 |J_{21}|
        \leq
        C
        [
        1+
%        (
        \1_{1 < \gamma \le 2}
        |x_0|_{L_\omega^{4\gamma}}^{3\gamma}
        +
        \1_{\gamma >2}
         |x_0|_{L_\omega^{6\gamma-4}}^{4\gamma-2}] \tau.
\end{align}
\end{lm}

\begin{proof}
A further decomposition is introduced for $J_{21}$:
\begin{align*}
 J_{21}
        & = 
        \sum_{n=0}^{N-1}\ee \int_{t_n}^{t_{n+1}}D u(T-s,\yy^n )
        ( b_\tau(\yy^n_n )- b(\yy^n_n )) \dd  s\\
        & \quad +
   \sum_{n=0}^{N-1}\ee \int_{t_n}^{t_{n+1}}D u(T-s,\yy^n )
        ( b(\yy^n_n )- b(\yy^n )) \dd  s \\
        &=:J_{211}+J_{212}.
\end{align*}
For $J_{211}$, we have the following decomposition:
\begin{align*}
        |J_{211}| 
        & \leq
        \Big[
         \sum_{n=0}^{N-N_1-2}  
          +
           \sum_{n=N-N_1-1}^{n=N-N_1-1}
         +
          \sum_{n=N-N_1}^{N-1}  
                     \Big]
       \ee
        \int_{t_n}^{t_{n+1}} \\
        &\qquad
        |
        D u
      (
        T-s,\yy^n
       ) 
       (
        b_\tau(\yy^n_n)
        -
        b(\yy^n_n)
      )
        |
       \dd  s\\
               &=:J_{211}^1+J_{211}^2+J_{211}^3.
\end{align*}
For  $J_{211}^1$, using  Lemma \ref{lm-Du}, Theorem \ref{Y-est},  and H\"older inequality, we have
\begin{align*}
        J_{211}^1
        & \leq
        C_{\star}
        \sum_{n=0}^{N-N_1-2}
       \ee
        \int_{t_n}^{t_{n+1}}
        e^{-\lambda (T-s-1)}
      (
        1+|\yy^n |
      )
     |
        b_\tau
        (
       \yy^n_n
        )
        -
        b
        (
       \yy^n_n
        )
     |
        %_{L^2(\Omega,\rr^d)}
       \dd  s\\
        & \leq
        C_{\star} \tau
        \sum_{n=0}^{N-N_1-2} 
        \int_{t_n}^{t_{n+1}}
        e^{-\lambda (T-s-1)}
        \sup_{r \ge 0}
        (
        (
        1+
        |Y_r|
        _{L_\omega^2
      }
        )
        |Y_r|^{\alpha_1}
        _{L_\omega^{2\alpha_1}}
        )
       \dd  s\\
        & \leq
        C_{\star}
        (
        1+
        |x_0|^{\alpha_1+1}
        _{L_\omega^{2\alpha_1}
        }
        )
        \tau.
\end{align*}
For $J_{211}^2$, one can derive
\begin{align*}
        J_{211}^2
        & \leq 
       \ee
        \int_{(N-N_1-1) \tau}^{N \tau-1}
        |
        D u
     (T-s,\yy^{N-N_1-1})
        (
        b_\tau
        (
       \yy^{N-N_1-1}_{N-N_1-1}
       )
        -
       b(
       \yy^{N-N_1-1}_{N-N_1-1}
      )
        )
        |
       \dd  s\\
        & \quad +
       \ee
        \int_{N \tau-1}^{(N-N_1) \tau} 
        |
        D u(T-s,\yy^{N-N_1-1})
        (
        b_\tau
       (
       \yy^{N-N_1-1}_{N-N_1-1})\ -
        b(
       \yy^{N-N_1-1}_{N-N_1-1} )
        )
        | \dd s\\
        & \leq
        C
        (
        1+
        |x_0|^{\alpha_1+1}
        _{L_\omega^{2\alpha_1}
       }
        )
        \tau^2.
\end{align*}
Similarly, we get
\begin{align*}
        J_{211}^3
         \leq
        C
        (
        1+
        |x_0|^{\alpha_1+1}
        _{L_\omega^{2\alpha_1}
       }
        )
        \tau.
\end{align*}
Putting all the above three estimates together, we obtain 
\begin{align*}
       |J_{211}|
         \leq
        C
        (
        1+
        |x_0|^{\alpha_1+1}
        _{L_\omega^{2\alpha_1}
       }
        )
        \tau.
\end{align*}
We  now  estimate $J_{212}$ as follows:
\begin{align*}
        |J_{212}|
        & \leq |
        \sum_{n=0}^{N-1} 
       \ee
        \int_{t_n}^{t_{n+1}}
        D u
        (
        T-s,\yy^n_n
        )
        (
        b
        (
       \yy^n_n
        )
        -
        b
        (
       \yy^n)
        ) 
       \dd  s
        |\\
        & \quad + 
        |
        \sum_{n=0}^{N-1} 
       \ee
        \int_{t_n}^{t_{n+1}}
        (
        D u
        (
        T-s,\yy^n
        )
        -
        D u
        (
        T-s,\yy^n_n
        )
        )
        (
        b
        (
       \yy^n_n
        )
        -
        b
        (
       \yy^n)
        ) 
       \dd  s
        |\\
        &=:J^{1}_{212}+J^{2}_{212}.
\end{align*}
For $J^{1}_{212}$, we get
\begin{align*}
        &\quad J^{1}_{212}\\
        & \leq
        [
         \sum_{n=0}^{N-N_1-2}  
          +
           \sum_{n=N-N_1-1}^{n=N-N_1-1}
         +
          \sum_{n=N-N_1}^{N-1}  
                     ]
                     | \ee \int_{t_n}^{t_{n+1}}
        D u(T-s,\yy^n_n)( b( \yy^n_n )-b( \yy^n) ) \dd  s|\\
               &=:T_1+T_2+T_3.
\end{align*}
To bound on  $J^{1}_{212}$, we use the Taylor expansion to obtain 
\begin{align*}
       b(\yy^n (s))
        & =
       b(
       \yy^n_n) 
        +
        D b(\yy^n _n)
        (
       \yy^n(s)
        -
       \yy^n_n
        )
        +
       \RR _b
        (
       \yy^n(s),
       \yy^n_n
        )\\
        & =
        b(
       \yy^n_n) 
        +
        D b(\yy^n _n)
        (
        b_\tau
        (\yy^n _n)
        (s-t_n)\\
        & \quad +
        \sum_{j=1}^m
        \sigma_{j,\tau}
        (\yy^n (t_n))
        (W_s-W_n)
        )
         +
       \RR _b
        (
       \yy^n(s),
       \yy^n_n
        ),
\end{align*}
where
\begin{align*}
       \RR _b
        (
       \yy^n(s),
       \yy^n_n
        )
         :=
        \int_0^1
        (
        Db(\yy^n _n
        +
        r
        (\yy^n (s)-\yy^n _n)
      )
        -
        D b(\yy^n _n)
        )
        (
       \yy^n(s)
        -
       \yy^n_n
        ) 
       \dd r.
\end{align*}
Moreover, by \eqref{Db-v}, Theorem  \ref{thm-Y}, one can further apply  H\"older inequality to deduce
\begin{align*}
  &|
   \RR _b
    (
   \yy^n(s),
   \yy^n_n
   )
    |
    _{L_\omega^2}\\
    & \leq 
    C 
    \int_0^1
    |
    (
    1+
   |
    r\yy^n(s)
    +
    (1-r)\yy^n_n |
    +
    |
   \yy^n_n
   |
    )^{\max \{0, \gamma-2\}}
    | 
   \yy^n(s)
    -\yy^n _n
   |^2
    |_{L_\omega^2}
   \dd  r \\
    & \leq 
        C
        (
        \1_{1 < \gamma \le 2}
       |
       \yy^n(s)-\yy^n _n|
        _{L_\omega^4}^2
        +
        \1_{\gamma >2}
         \int_0^1
      |
       (1+|r\yy^n(s)+(1-r)\yy^n_n
       |
        \nonumber\\
        & \qquad 
        +
      |
       \yy^n_n
       |
       )^{\gamma-2}
        | 
       \yy^n(s)-\yy^n _n
       |^2
       \big |_{L_\omega^2}
       \dd  r
        ) \\
        & \leq 
        C
        [
        1+
        (
        \1_{1 < \gamma \le 2}
        |x_0|_{L_\omega^{4\gamma}}^{2\gamma}
        +
        \1_{\gamma >2}
        |x_0|
        _{L_\omega^{6\gamma-4}}
        ^{3\gamma-2}
        )
        ] 
        \tau. 
        \end{align*}
For $T_1$,  we have
\begin{align*}
       T_1
         &=
        |
        \sum_{n=0}^{N-N_1-2}   \ee \int_{t_n}^{t_{n+1}} D u (T-s,\yy^n_n  )
        D b(\yy^n _n)b_\tau ( \yy^n_n)(s-t_n) \dd s \\
        &\qquad +
       \sum_{n=0}^{N-N_1-2}   \ee \int_{t_n}^{t_{n+1}} D u (T-s,\yy^n_n  )
       \RR _b(\yy^n,\yy^n_n) \dd  s |\\
        & \leq
        C_{\star}
        \sum_{n=0}^{N-N_1-2}
        \int_{t_n}^{t_{n+1}}
        e^{-\lambda (T-s-1)}
        (
        1+
        |\yy^n (t_n)|
        _{L_\omega^2}
        )\\
        & \qquad \times
        (
       |
        D b
        (
       \yy^n_n
       ) 
        b_\tau
       (
       \yy^n_n
       )
        \tau
        |_{L_\omega^2}
        +
       |
       \RR _b
      (
       \yy^n,
       \yy^n_n
       )
     |
        _{L_\omega^2}
        )
       \dd  s\\
        & \leq 
        C [
        1+
%        (
        \1_{1 < \gamma \le 2}
        |x_0|_{L_\omega^{4\gamma}}^{ 2\gamma}
        +
        \1_{\gamma >2}
        |x_0|_{L_\omega^{6\gamma-4}}^{3\gamma-2}
%        )
        ] \tau,
\end{align*}
where we derive from \eqref{Db-v}, \eqref{b-grow},  and Theorem \ref{thm-Y} that
\begin{align*}
        |
        D b
      (
       \yy^n_n
       ) 
        b_\tau
     (
       \yy^n_n
        )
      \tau
        |
        _{L_\omega^2}
          \leq
        C
        (
        1+
        |x_0|^{2\gamma-1}_{L_\omega^{4\gamma-2)}}
        )\tau.
\end{align*}

The proofs for $T_2$ and $T_3$ are similar. Then we get
\begin{align*}
        J^{1}_{212} 
        \leq
         C
        [
        1+
%        (
        \1_{1 < \gamma \le 2}
        |x_0|_{L_\omega^{4\gamma}}^{ 2\gamma}
        +
        \1_{\gamma >2}
        |x_0|_{L_\omega^{6\gamma-4}}^{3\gamma-2}
%        )
        ] \tau.
\end{align*}
Applying the Taylor expansion to $ D u(t,\cdot)$ gives  
\begin{align*}
        J^{2}_{212}
        & \leq  
         \Big[
         \sum_{n=0}^{N-N_1-2}  
          +
           \sum_{n=N-N_1-1}^{n=N-N_1-1}
         +
          \sum_{n=N-N_1}^{N-1}  
                     \Big]
        \Big|
       \ee
        \int_{t_n}^{t_{n+1}}
        \int_0^1
        D^2 u(
        T-s, v_2(\bar{r})) \\
        &\qquad
     (
       \yy^n
        -
       \yy^n_n,
        b
        (
       \yy^n_n
        )
        -
        b
       (
       \yy^n)
        ) 
       \dd  \bar{r}
       \dd  s \Big|\\
        &=:T_4+T_5+T_6,
\end{align*}
where $v_2(\bar{r}):=\yy^n _n+\bar{r}(\yy^n (s)-\yy^n_n)$.
By \eqref{D2u-est+}, we have
\begin{align*}
        T_4
        & \leq
        C_{\star}
        \sum_{n=0}^{N-N_1-2} 
       \ee
        \int_{t_n}^{ t_{n+1} } 
        e ^ {-\lambda(T-s-1) } 
        [
        1+
        \1_{1 < \gamma \le 2}
        |v_2(\bar{r})|
        +
        \1_{\gamma >2}
        |v_2(\bar{r})|^{\gamma-1}
        ]\\
        & \qquad \times
       |
       \yy^n-\yy^n _n
        |
        \cdot
        |
        b(\yy^n _n)
        -
        b(\yy^n)
        |
       \dd  \bar{r}
       \dd  s.
\end{align*}

%%%%%%%%%%%%%%%%%%%
\iffalse
Following the same arguments as used in the estimate of $T_4$, one derives from \eqref{D2u-est} in Theorem \ref{lm-Du} that
\begin{align*}
        T_5
        & \leq
        C
        \sum_{n=N-N_1}^{N-1} 
       \ee
        \int_{t_n}^{ t_{n+1} } 
        \int_{0}^{1}
        \frac{1}{\sqrt{T-s}}
        [
        1+
        (
        \1_{1 < \gamma \le 2}
        |v_2(\bar{r})|
        +
        \1_{\gamma >2}
        |v_2(\bar{r})|^{\gamma-1}
        )
        ]\\
        & \qquad \qquad \cdot
        |
       \yy^n(s)-\yy^n (t_n)
        |
        \cdot
        |
        f(\yy^n (t_n))
        -
        f(\yy^n (s))
        |
       \dd  \bar{r} 
       \dd  s.
\end{align*}
\fi
%%%%%%%%%%%%%%%%%%%%%%%%%

To proceed further, we shall consider two cases for   $\gamma$.

(1) In the case $1 < \gamma \leq2$, by H\"older inequality, Lemmas \ref{lm-Du},  \ref{lm-yy},  and Theorem \ref{thm-Y},  we obtain
\begin{align*}
        T_4
        & \leq
        C_{\star}
        \sum_{n=0}^{N-N_1-2} 
       \ee
        \int_{t_n}^{ t_{n+1} } 
        e ^ {-\lambda(T-s-1) }
        (
        1+
        \sup_{r\ge 0}
        |Y_r|_{L_\omega^{3\gamma}}
        )\\
        & \qquad \times
        |
       \yy^n
        -
       \yy^n_n
      |_{L_\omega^{3}}
        |
        b(\yy^n _n)
        -b(\yy^n )
       |
        _{L_\omega^{3\gamma/(2\gamma-1)}} 
       \dd  s\\
        & \leq 
        C (1+
        |x_0|^{3\gamma}
        _{L_\omega^{3\gamma}}
        ) \tau.
\end{align*}

%%%%%%%%%%%%%%%%%%%%%%
\iffalse
Concerning $T_5$, one similarly deduces
\begin{align}
    T_5 \leq 
    C
    (1+
    |x_0|^{2\gamma+1}
    _{L_\omega^{4\gamma+1}}
   ) 
    \tau,
\end{align}
where 
\begin{align}
    \sum_{n=N-N_1}^{N-1} 
    \int_{t_n}^{t_{n+1}} 
    \frac{1}{\sqrt{T-s}} 
   \dd  s 
    \leq 
    \int_0^1 \frac{1}{\sqrt{t}} 
   \dd  t
    = 2.
\end{align}
\fi
%%%%%%%%%%%%%%%%%

(2) In the case $\gamma >2$,  we have
\begin{align*}
        T_4
        & \leq
        C_{\star}
        \sum_{n=0}^{N-N_1-2} 
       \ee
        \int_{t_n}^{ t_{n+1} } 
        e ^ {-\lambda(T-s-1) }
        (
        1+
        \sup_{r \ge 0}
        |Y_r|^{\gamma-1}_{L_\omega^{4\gamma-2}}
        )\\
        & \qquad \times
       |
       \yy^n
        -
       \yy^n_n
       |_{L_\omega^{(4\gamma-2)/\gamma}}
        |
        b(\yy^n _n)
        -b(\yy^n )
        |_{L_\omega^{(4 \gamma -2)/(2\gamma-1)}} 
       \dd  s\\
        & \leq
        C        (1+
        |x_0|^{4\gamma-2}
        _{L_\omega^{4\gamma-2}}
        ) \tau.  
\end{align*}
Given the above estimations, we get
\begin{align*}
    T_4 \leq
    C [
    1+
%    (
    \1_{1 < \gamma \le 2}
    |x_0|_{L_\omega^{3\gamma}}^{3\gamma}
    +
    \1_{\gamma >2}
    |x_0|_{L_\omega^{4\gamma-2}}^{4\gamma-2}
%    )
    ] \tau.
\end{align*}
Since the estimates of $T_5$  and $T_6$ follow similar arguments, we omit them for brevity.
%
%
%%%%%%%%%%%%%%%%%%%
\iffalse
\begin{align}
    
        &
        T_6\\
        & \leq
        \Big|
       \ee 
        \Big[
        \int_{\left(N-N_1-1\right) h}^{N h-1} 
        \int_0^1 
        D^2 u\left(T-s, v_2^{N-N_1-1}(\bar{r})\right)
        (
       \yy^{N-N_1-1}(s)
        -
       \yy^{N-N_1-1}\left(t_{N-N_1-1}\right),\\
        &  \qquad \qquad
        f\left(\yy  ^{N-N_1-1}\left(t_{N-N_1-1}\right)\right)-f\left(\yy  ^{N-N_1-1}(s)\right)
        ) 
       \dd  \bar{r}\dd  s
        \Big] 
        \Big|\\
        & \quad +
        \Big|
       \ee
        \Big[
        \int_{N h-1}^{\left(N-N_1\right) h} 
        \int_0^1 
        D^2 u\left(T-s, v_2^{N-N_1-1}(\bar{r})\right)
        (
       \yy^{N-N_1-1}(s)
        -
       \yy^{N-N_1-1}\left(t_{N-N_1-1}\right),\\
        & \qquad \qquad
        f\left(\yy  ^{N-N_1-1}\left(t_{N-N_1-1}\right)\right)-f\left(\yy  ^{N-N_1-1}(s)\right)
        ) 
       \dd  \bar{r} 
       \dd  s
        \Big]
        \Big| \\
        & \leq  
        C
        \left(
        1+
        |x_0|^{4\gamma-1}_{L^{4 \gamma}\left(\Omega, \rr^d\right)}
        \right)
        h^{3/2},
\end{align}
where 
\begin{align*}
    \int_{N h-1}^{\left(N-N_1\right) h} 
    \frac{1}{\sqrt{T-s}}
   \dd  s
    =
    \int_{N_1 h}^1
    \frac{1}{t}
   \dd  t
    \leq
    C \sqrt{t.}
\end{align*}
\fi
%%%%%%%%%%%%%%%%%%%%%%%%%%%%
%
%
Combining all the estimates of $T_4$, $T_5$, and $T_6$, we obtain
\begin{align*}
    J^{2}_{212}
    \leq
    C
    [
    1+
%    (
    \1_{1 < \gamma \le 2}
    |x_0|_{L_\omega^{3\gamma}}^{3\gamma}
    +
    \1_{\gamma >2}
    |x_0|_{L_\omega^{4\gamma-2}}^{4\gamma-2}
%    )
    ] \tau.
\end{align*}
By the estimates of $J_{212}^1$ and $J_{212}^2$, we conclude \eqref{J21-est}.
\end{proof}

\begin{lm}
\label{lm-J22}
Let Assumptions \ref{A1}-\ref{A5} hold. 
%Let $\left\{\yy (s)\right\}_{s \in\left[t_n, t_{n+1}\right]}$ and $\{Y_n\}_{n \in \{0,1,\cdots,N-1\}}$ be defined as \eqref{yy} and \eqref{MEM}, $n \in\{0,1, \cdots, N-1\}, N \in \nn$. 
For any $x_0 \in L^{l_1 \vee l_3 \vee l_5}(\Omega; \rr^d)$, where $l_1 := \max\{4\alpha_1,\alpha_1+\gamma+2,2\alpha_2,4\gamma+1\},~ l_3 := \max\{2 \gamma+\alpha_1,2\alpha_1+\gamma-1,2\alpha_2,5\gamma\},~ l_5 := \max\{2 \gamma+\alpha_1,2\alpha_1+\gamma-1,2\alpha_2,7\gamma-3\},$
there exists a positive constant $C$ such that
\begin{align}
\label{J22-est}
|J_{22}|
         \leq
        C
        [
        1+
        \1_{1 < \gamma \le 2}
        |x_0|
        _{L_\omega^{l_1}}
        ^{l_2}+
        \1_{\gamma \in(2,3)}
        |x_0|_{L_\omega^{l_3}}^{l_4}
        +\1_{\gamma >3}
        |x_0|_{L_\omega^{l_5}}^{l_6}
        ] \tau,
\end{align}
where $ l_2 := \max\{2\alpha_1+1, (\gamma+2\alpha_1+3)/2,\alpha_2+1,(7\gamma+1)/2\},~ l_4 := \max\{\gamma+2\alpha_1-1,(3 \gamma+2\alpha_1-1)/2,\alpha_2+1,(7\gamma-1)/2,~ l_6 := \max\{\gamma+2\alpha_1-1,(3 \gamma+2\alpha_1-1)/2,\alpha_2+1,4\gamma-2\}.$
\end{lm}

\begin{proof}
For  $J_{22}$, we have the following decomposition:
\begin{align*}
        J_{22}
        & =
        \frac{1}{2}         \sum_{n=0}^{N-1} 
        \sum_{j=1}^m
       \ee
        \int_{t_n}^{t_{n+1}}
        D^2 u
        (
        T-s,\yy^n
       )
        (
        \sigma_{j,\tau}
       (
       \yy^n_n
       ), 
        \sigma_{j,\tau}
      (
       \yy^n_n
    )
        )  \\
        & 
        \qquad  -
        D^2 u
       (
        T-s,\yy^n
    )
        (
        \sigma_{j}
      (
       \yy^n_n
        ), 
        \sigma_{j}
     (
       \yy^n_n
      )
        )
       \dd  s\\
        & \quad +
        \frac{1}{2}
         \sum_{n=0}^{N-1} 
        \sum_{j=1}^m
       \ee
        \int_{t_n}^{t_{n+1}}
        D^2 u
     (
        T-s,\yy^n
       )
        (
        \sigma_{j}
      (
       \yy^n_n
       ), 
        \sigma_{j}
       (
       \yy^n_n
      )
        )  \\
        &
        \qquad -
        D^2 u
      (
        T-s,\yy^n
      )
        (
        \sigma_{j}
      (
       \yy^n
       ), 
        \sigma_{j}
      (
       \yy^n
      )
        )
       \dd  s \\
        &=:J_{221}+J_{222}.
\end{align*}
%For the term $J_{221}$,  by applying the following equality, 
Noting that, for any matrix $ U \in \rr^{d \times d}$ and any  $a, b \in \rr^d$,
$$
a^T U a-b^T U b=(a-b)^T U(a-b)+(a-b)^T U b+b^T U(a-b),
$$
one has a further decomposition of $J_{221}$ as follows:
\begin{align*}
        J_{221}
        & =    
        \frac{1}{2} 
         \sum_{n=0}^{N-1} 
        \sum_{j=1}^m
       \ee
        \int_{t_n}^{t_{n+1}}
        D^2 u(
        T-s,\yy^n )
       (
        \sigma_j(\yy^n_n)        - 
        \sigma_{j,\tau}(\yy^n_n),
        \sigma_j(\yy^n_n)
        -
        \sigma_{j,\tau}(\yy^n_n)
      \big )
       \dd s\\
         & \quad +     
         \frac{1}{2}
          \sum_{n=0}^{N-1} 
        \sum_{j=1}^m
        \ee
         \int_{t_n}^{t_{n+1}}
         D^2 
         u(T-s,\yy^n  )
         (
         \sigma_j(\yy^n_n)
         -
         \sigma_{j,\tau}(\yy^n_n),
         \sigma_{j,\tau}(\yy^n _n)
         )
        \dd s\\
         & \quad +             \frac{1}{2}
           \sum_{n=0}^{N-1} 
        \sum_{j=1}^m
        \ee
         \int_{t_n}^{t_{n+1}} 
         D^2 u(T-s,\yy^n)
         (
         \sigma_j(\yy^n _n),
         \sigma_j(\yy^n _n)
         -
         \sigma_{j,\tau}(\yy^n  _n
         )
        \dd s \\
         &=:J_{221}^1+J_{221}^2+J_{221}^3.
\end{align*}
Similar to what was done before, we have
\begin{align*}
|J_{221}^1|&\le
         \frac{1}{2}
         [
         \sum_{n=0}^{N-N_1-2}  
          +
           \sum_{n=N-N_1-1}^{n=N-N_1-1}
         +
          \sum_{n=N-N_1}^{N-1}  
                     ]
         \sum_{j=1}^m 
        \ee
         \int_{t_n}^{t_{n+1}}
         |
         D^2 u
         (
         T-s,\yy^n 
       ) \\
       &\qquad
         (
         \sigma_j(\yy^n_n)         - 
         \sigma_{j,\tau}(\yy^n_n),
         \sigma_j(\yy^n_n)
         -
         \sigma_{j,\tau}(\yy^n_n)
         )
         |
        \dd s\\
        &=:R_1+R_2+R_3.
\end{align*}

\iffalse
Similar to what was done before, we split the time interval $[0, T)$ into 
$[0,T-N_1\tau-\tau), [T-N_1\tau-\tau,T-N_1\tau)$ and
$[T-N_1\tau,T)$ as follows:
\begin{align*}
        &\quad R_{31}\\
        & \leq  
         \frac{1}{2}
         \sum_{n=0}^{N-N_1-2}  
         \sum_{j=1}^m 
        \ee
         \int_{t_n}^{t_{n+1}}
         |
         D^2 u
         (
         T-s,\yy^n 
       )
         (
         \sigma_j(\yy^n_n)         - 
         \sigma_{j,\tau}(\yy^n_n),
         \sigma_j(\yy^n_n)
         -
         \sigma_{j,\tau}(\yy^n_n)
         )
         |
        \dd s\\
         & \quad +
         \frac{1}{2}
         \sum_{n=N-N_1}^{N-1}  
         \sum_{j=1}^m 
        \ee
         \int_{t_n}^{t_{n+1}}
         |
         D^2 u
      (
         T-s,\yy^n           )
         (
         \sigma_j(\yy^n  _n)         - 
         \sigma_{j,\tau}(\yy^n  _n),
         \sigma_j(\yy^n _n)
         -
         \sigma_{j,\tau}(\yy^n _n)
         )
         |
        \dd s \\
         & \quad +          \frac{1}{2} 
         \sum_{j=1}^m
        \ee
         \int_{(N-N_1-1)\tau}^{(N-N_1)\tau}
         |
         D^2 u
        (
         T-s,\yy^{N-N_1-1} )
         (
         \sigma_j(\yy^{N-N_1-1}  _{N-N_1-1})-
         \sigma_{j,\tau}(\yy^{N-N_1-1}  _{N-N_1-1}), \\
         &\qquad 
                  \sigma_j(\yy^{N-N_1-1}  _{N-N_1-1})
         -
         \sigma_{j,\tau}(\yy^{N-N_1-1}  _{N-N_1-1})
         )
         |
        \dd s\\
        &=:R_{311}+R_{312}+R_{313}.
\end{align*}
\fi

For  $R_1$, by Lemma \ref{lm-Du},  Theorem \ref{Y-est}, and  Assumption \ref{A4}, we infer
\begin{align*}
        R_1
        & \leq
        C_{\star}
        \sum_{n=0}^{N-N_1-2}  
        \sum_{j=1}^m 
       \ee
        \int_{t_n}^{t_{n+1}}
        e^{-\lambda(T-s-1)}
        [
        1+
%        (
        \1_{1 < \gamma \le 2}(\yy^n  )|\\
                & \qquad +
        \1_{\gamma >2}
        |\yy^n  |^{\gamma-1}
%        )
        ]
       |
        \sigma_j(\yy^n  _n)
        - 
        \sigma_{j,\tau}(\yy^n _n)
        |^2
       \dd  s\\
        & \leq 
        C_{\star}
        \sum_{n=0}^{N-N_1-2}  
        \sum_{j=1}^m 
        \int_{t_n}^{t_{n+1}}
        e^{-\lambda(T-s-1)}
        [
        1+
        \1_{1 < \gamma \le 2}
        |\yy^n |
        _{L_\omega^2}
        ]\\
        & \qquad \qquad \times
       |
        \sigma_j(\yy^n _n)
        - 
        \sigma_{j,\tau}(\yy^n _n)
      |^{2}_{L_\omega^{4}}
       \dd s\\
        & \quad +
        C_{\star}
        \sum_{n=0}^{N-N_1-2}  
        \sum_{j=1}^m 
        \int_{t_n}^{t_{n+1}}
        e^{-\lambda(T-s-1)}
        [
        1+
        \1_{\gamma \in(2,\infty)}
        |\yy^n |^{\gamma-1}
        _{L_\omega^{\gamma+2\alpha_1-1}}
        ]\\
        & \qquad  \times
        |
        \sigma_j(\yy^n  _n)
        - 
        \sigma_{j,\tau}(\yy^n _n)
       |^{2}
        _{L_\omega^{{(\gamma+2\alpha_1-1)/\alpha_1}}}
       \dd s \\
        & \leq
        C_{\star}
        [
        1+
%        (
        \1_{1 < \gamma \le 2}
        |x_0|^{2\alpha_1+1}
        _{L_\omega^{4\alpha_1}}
        +
        \1_{\gamma >2}
        |x_0|^{2\alpha_1+\gamma-1}_{L_\omega^{2\alpha_1+\gamma-1}}
%        )
        ]\tau^2.
\end{align*}
The estimate for $R_2$ is similar to $J_{211}^2$, and we omit it here.
For $R_3$, one deduces
\begin{align*}
        R_3
        & \leq
        C
         \sum_{n=N-N_1}^{N-1}  
         \sum_{j=1}^m 
        \ee
         \int_{t_n}^{t_{n+1}}
         \frac{1}{\sqrt{T-s}}
         [
        1+
%        (
        \1_{1 < \gamma \le 2}|\yy^n |\\
        & \qquad +
        \1_{\gamma >2}
        |\yy^n  |^{\gamma-1}
%        )
        ]
       |
        \sigma_j(\yy^n  _n)
        - 
        \sigma_{j,\tau}(\yy^n  _n)
       |^2
       \dd  s\\
         & \leq
         C
        [
        1+
%        (
        \1_{1 < \gamma \le 2}
        |x_0|^{2\alpha_1+1}
        _{L_\omega^{4\alpha_1}}
        +
        \1_{\gamma >2}
        |x_0|^{2\alpha_1+\gamma-1}_{L_\omega^{2\alpha_1+\gamma-1}}
%        )
        ] \tau^2,
\end{align*}
where
\begin{align*}
    \sum_{n=N-N_1}^{N-1}
    \int_{t_n}^{t_{n+1}}
    \frac{1}{\sqrt{T-s}}
   \dd  s
    \leq 
    \int_0^1 \frac{1}{\sqrt{t}} 
   \dd  t
    = 2.
\end{align*}
%
%
%
%%%%%%%%%%%%%%%%%%%%%%%
\iffalse
\begin{align}
\label{R313-est}
    R_{3,1,3}
    \leq 
    C
    [
        \1_{1 < \gamma \le 2}
        (1+
        |x_0|^{2\alpha_1+1}
        _{L_\omega^{4\alpha_1}}
        )
        +
        \1_{\gamma >2}
        (1+
        |x_0|^{2\alpha_1+\gamma-1}_{L_\omega^{2\alpha_1+\gamma-1}}
        )
        ]  tau^{5/2}.
\end{align}
\fi
%%%%%%%%%%%%%%%%%%%%%%%%%
%
%
Thus, we have
\begin{align*}
    |J_{221}^1|
    \leq
    C
    [
    1+
%   (
    \1_{1 < \gamma \le 2}
    |x_0|^{2\alpha_1+1}
    _{L_\omega^{4\alpha_1}}
    +
    \1_{\gamma >2}
    |x_0|^{2\alpha_1+\gamma-1}
    _{L_\omega^{2\alpha_1+\gamma-1}}
%   )
    ]\tau^2.
\end{align*}
The terms $J_{221}^2$ and $J_{221}^3$ can be treated similarly.
Therefore, we get
\begin{align*}
        |J_{221}^2|
        & \leq  
         \frac{1}{2} 
          [
         \sum_{n=0}^{N-N_1-2}  
          +
           \sum_{n=N-N_1-1}^{n=N-N_1-1}
         +
          \sum_{n=N-N_1}^{N-1}  
                     ]
         \sum_{j=1}^m 
        \ee
         \int_{t_n}^{t_{n+1}}
         |
         D^2 
         u(T-s,\yy^n  ) \\
         &\qquad
         (
         \sigma_j(\yy^n  _n)-
         \sigma_{j,\tau}(\yy^n _n),
         \sigma_{j,\tau}(\yy^n_n)
         )
         |
        \dd s\\
        &=:R_4+R_5+R_6.
\end{align*}
From   Lemma \ref{lm-Du}, we see
\begin{align*}
        R_4
        & \leq
        C_{\star}
        \sum_{n=0}^{N-N_1-2}  
        \sum_{j=1}^m 
       \ee
        \int_{t_n}^{t_{n+1}}
        e^{-\lambda(T-s-1)}
        [
        1+
        \1_{1 < \gamma \le 2}|\yy^n |+
        \1_{\gamma >2} |\yy^n  |^{\gamma-1}
        ]\nonumber\\
        & \qquad 
        \times
      |
        \sigma_j(\yy^n_n)
        - 
        \sigma_{j,\tau}(\yy^n_n)
       |
        |\sigma_{j,\tau}(\yy^n_n)|
       \dd s.
\end{align*}
In the following, we proceed to the estimate of $R_4$, which will be divided into two cases depending on the ranges of $\gamma$.

(1) In the case $1 < \gamma \leq2$,  using Lemma \ref{lm-yy}, Theorem \ref{thm-Y},   and H\"older inequality, we  obtain
\begin{align*}
    R_4
    & \leq
    C_{\star}
    \sum_{n=0}^{N-N_1-2}  
    \sum_{j=1}^m 
   \ee
    \int_{t_n}^{t_{n+1}}
    e^{-\lambda(T-s-1)}
    (
    1+
    \sup_{r \ge 0}
    |Y_r|
    _{L_\omega^{\alpha_1+\gamma+2}}
    )\\
    & \qquad \times
  |
    \sigma_j(\yy^n_n)
    - 
    \sigma_{j,\tau}(\yy^n _n)
 |
    _{L_\omega^{(\alpha_1+\gamma+2)/\alpha_1}
   }
   |
    \sigma_{j,\tau}(\yy^n_n)
   |
    _{L_\omega^{(\alpha_1+\gamma+2)/(\gamma+1)}} 
   \dd  s\\
    & \leq 
    C (1+
    |x_0|
    ^{(2\alpha_1+\gamma+3)/2}
    _{L_\omega^{\alpha_1+\gamma+2}}
    ) \tau.
\end{align*}

(2) In the case $\gamma >2$,  we derive
\begin{align*}
    R_4
    & \leq
    C_{\star}
    \sum_{n=0}^{N-N_1-2}  
    \sum_{j=1}^m 
   \ee
    \int_{t_n}^{t_{n+1}}
    e^{-\lambda(T-s-1)}
    (
    1+
    \sup_{r \ge 0}
    |Y_r|^{\gamma-1}
    _{L_\omega^{\alpha_1+2\gamma}}
    )\\
    & \qquad \times
  |
    \sigma_j(\yy^n_n)
    - 
    \sigma_{j,\tau}(\yy^n_n)
    |
    _{L_\omega^{(\alpha_1+2\gamma)/\alpha_1}
   }
 |
    \sigma_{j,\tau}(\yy^n_n)
   |
    _{L_\omega^{(\alpha_1+2\gamma)/(\gamma+1)}} 
   \dd  s\\
    & \leq 
    C (1+
    |x_0|
    ^{(2\alpha_1+3\gamma-1)/2}
    _{L_\omega^{\alpha_1+2\gamma}}
    ) \tau.
\end{align*}
As a consequence, in any case, we have
\begin{align*}
    R_4
    & \leq
    C \Big[
    1+
    \1_{\gamma\in(1,2]}
    |x_0|
    _{L_\omega^{\alpha_1+\gamma+2}}
    ^{ (2\alpha_1+\gamma+3)/2}
    +
    \1_{\gamma >2}
    |x_0|_{L_\omega^{\alpha_1+2 \gamma}}^{(2\alpha_1+3 \gamma-1)/2}
    \Big]  
    \tau.
\end{align*}
%
%%%%
%%%%%%%%%%%%%%%%%%%%%%%
\iffalse
For $T_{11}$, we have
\begin{align}
        T_{11} 
        & \leq
         C
         \sum_{n=N-N_1}^{N-1}  
         \sum_{j=1}^m 
        \ee
         \Big[ 
         \int_{t_n}^{t_{n+1}}
         \frac{1}{\sqrt{T-s}}
         (
         1+|\yy^n (s)|
         )
         [
         \1_{1 < \gamma \le 2}\\
         & \qquad +
         \1_{\gamma >2}
         (1+
        |\yy^n (s)|^{\gamma-2})
        ]
        |
        \sigma_j(\yy^n (t_n))
        - 
        \sigma_{j,\tau}(\yy^n (t_n))
        |
        \cdot
        |\sigma_{j,\tau}(\yy^n (t_n))|
       \dd s 
        \Big]\\
         & \leq
         C 
         \left[
         \1_{1 < \gamma \le 2}
         (
         1+
         |x_0|
         _{L^{4\alpha_1}
         (\Omega,\rr^d
         )}
         ^{ (\gamma+2\alpha_1+2)/2}
         )
          +
         \1_{\gamma >2}
        (
        1+|x_0|_{L^{4 \gamma}(\Omega, \rr^d)}^{(2 \gamma+2\alpha_1-1)/2})
        \right] h.
\end{align}
For the estimate of $T_{12}$, we can similarly obtain that 
\begin{align}
    T_{12} \leq C 
         \left[
         \1_{1 < \gamma \le 2}
         (
         1+
         |x_0|
         _{L^{4\alpha_1}
         (\Omega,\rr^d
         )}
         ^{ (\gamma+2\alpha_1+2)/2}
         )
          +
         \1_{\gamma >2}
        (
        1+|x_0|_{L^{4 \gamma}(\Omega, \rr^d)}^{(2\alpha_1+2 \gamma-1)/2})
        \right] h^{3/2}.
\end{align}
\fi
%%%%%%%%%%%%%%%%%%%%%%%%%%%%%
%
Due to similar techniques and length limitations, we omit the estimation details for $R_5$ and $R_6$.
Armed with the estimate of $R_4$,  we obtain
\begin{align*}
    |J_{221}^2|
    \leq 
    C
    \Big[
    1+
    \1_{\gamma\in(1,2]}
    |x_0|
    _{L_\omega^{\alpha_1+\gamma+2}}
    ^{ (2\alpha_1+\gamma+3)/2}
    +
    \1_{\gamma >2}
    |x_0|_{L_\omega^{\alpha_1+2 \gamma}}^{(2\alpha_1+3 \gamma-1)/2}
    \Big]  
    \tau.
\end{align*}
Then we obtain
\begin{align*}
|J_{221}|
         \leq
        C
        \Big [
        1+
        \1_{1 < \gamma \le 2}
        |x_0|
        _{L_\omega^{l_1}}
        ^{l_2}+
        \1_{\gamma >2}
        |x_0|_{L_\omega^{l_3}}^{l_4}
        \Big ] 
        \tau.
\end{align*}
For $J_{222}$, we make a decomposition as follows:
\begin{align*}
        J_{222}
        & =
         \frac{1}{2} 
         \sum_{n=0}^{N-1}  \sum_{j=1}^m 
        \ee
         \int_{t_n}^{t_{n+1}}
         D^2 u
    (
         T-s,\yy^n 
    )
         (
         \sigma_j(\yy^n )
         - 
         \sigma_j(\yy^n _n),
         \sigma_j(\yy^n )
         -
         \sigma_j(\yy^n _n)
         )
        \dd s \\
         & \quad + 
         \frac{1}{2} 
         \sum_{n=0}^{N-1}  \sum_{j=1}^m 
        \ee 
         \int_{t_n}^{t_{n+1}}
         D^2 
         u(T-s,\yy^n )
         (\sigma_j(\yy^n )
         -
         \sigma_j(\yy^n _n),
         \sigma_j(\yy^n _n)
         )
        \dd s \\
         & \quad +  
         \frac{1}{2} 
         \sum_{n=0}^{N-1}  \sum_{j=1}^m 
        \ee
         \int_{t_n}^{t_{n+1}} 
         D^2 u(T-s,\yy^n )
         (
         \sigma_j(\yy^n _n),
         \sigma_j(\yy^n )
         -
         \sigma_j(\yy^n _n)
         )
        \dd s \\
         &=:J^{1}_{222}+J^{2}_{222}+J^{3}_{222}.
\end{align*}
Using Lemmas  \ref{lm-Du},  \ref{lm-yy},   Theorem \ref{thm-Y},   and  H\"older inequality shows
\begin{align*}
   | J^{1}_{222} |
    \leq
    C
    [
    1+
%    (
    \1_{1 < \gamma \le 2}
    |x_0|_{L_\omega^{3\gamma}}^{3\gamma}
    +
    \1_{\gamma >2}
    |x_0|_{L_\omega^{4\gamma-2}}^{4\gamma-2}
%    )
    ] \tau.
\end{align*}
%
%
%%%%%%%%%%%%%%%%%%%%%%%
\iffalse
\begin{align}
    
        &
        J^{(1)}_{2,2,1}\\
        & \leq
        C_{\star}
        \sum_{n=0}^{N-N_1-2}  
        \sum_{j=1}^m 
       \ee
        \Big[ 
        \int_{t_n}^{t_{n+1}}
        e^{-\lambda(T-s-1)}
        (
        1+|\yy^n (s)|
        )
        |
        \sigma_j(\yy^n (s))
        - 
        \sigma_j(\yy^n (t_n))
        |^2
       \dd s 
        \Big]\\
        & \quad +
        C
         \sum_{n=N-N_1}^{N-1}  
         \sum_{j=1}^m 
        \ee
         \Big[ 
         \int_{t_n}^{t_{n+1}}
         \frac{1}{\sqrt{T-s}}
         (
         1+|\yy^n (s)|
         )
         |
         \sigma_j(\yy^n (s))
         - 
         \sigma_j(\yy^n (t_n))
        |^2
        \dd s 
         \Big]\\
    
\end{align}
\fi
%%%%%%%%%%%%%%%%%%%%%%%%%%
%
%
For  $J^{2}_{222}$, we first use the triangle inequality to get
\begin{align*}
        J^{2}_{222}
        & = \frac{1}{2} 
         \sum_{n=0}^{N-1}  \sum_{j=1}^m 
        \ee
         \int_{t_n}^{t_{n+1}}
         D^2 
         u(T-s,\yy^n _n)
         (\sigma_j(\yy^n )-\sigma_j(\yy^n _n),
         \sigma_j(\yy^n _n)
         )
        \dd s 
         \\
         & \quad +
        \frac{1}{2} 
         \sum_{n=0}^{N-1}  \sum_{j=1}^m 
        \ee
         \int_{t_n}^{t_{n+1}}
         (D^2 
         u(T-s,\yy^n )
         -
         D^2 
         u(T-s,\yy^n _n)
         )\\
         & \qquad \times
        (\sigma_j(\yy^n )
         -\sigma_j(\yy^n _n),
         \sigma_j(\yy^n _n)
         )
        \dd s \\
         &=:R_7+R_8.
\end{align*}
By the Taylor expansion, we have
\begin{align*}
        \sigma_j
      (\yy^n (s))
        & = 
        \sigma_j
        (
       \yy^n_n
       )
        +
        D \sigma_j
     (
       \yy^n_n)
     (\yy^n (s)-\yy^n _n)
        +
       \RR _{\sigma_j}
       (
       \yy^n(s),
       \yy^n_n) \\
        & = 
        \sigma_j
        (
       \yy^n_n
        )
        +
        D \sigma_j
        (
        Y_n)
        (
        b_\tau(\yy^n _n)(s-t_n)\\
        & \quad +
        \sigma_{j,\tau}
        (
       \yy^n_n
        )
        (W_s-W_n)
        )
        +
       \RR _{\sigma_j}
        (
       \yy^n(s),
       \yy^n_n
        ),
\end{align*}
where we denote
\begin{align*}
       \RR _{\sigma_j}
        (
       \yy^n(s),
       \yy^n_n
        ) :=
        \int_0^1
        (
        D \sigma_j
        (
       \yy^n_n)
        +
        r
        (
       \yy^n(s)
        -\yy^n _n
        )
        -
        D \sigma_j
        (
       \yy^n_n
        )
        )
        (
       \yy^n(s)
        -
       \yy^n_n
        ) 
       \dd r.
\end{align*}
Similarly, we treat $R_7$ as follows:
\begin{align*}
        |R_7|
        & \leq    
        \frac{1}{2} 
        \Big [
         \sum_{n=0}^{N-N_1-2}  
          +
           \sum_{n=N-N_1-1}^{n=N-N_1-1}
         +
          \sum_{n=N-N_1}^{N-1}  
                     \Big ]
  \sum_{j=1}^m 
        \Big|
       \ee
        \int_{t_n}^{t_{n+1}}
        D^2 
        u(T-s,\yy^n _n) \\
        &\qquad
        (\sigma_j(\yy^n )-
        \sigma_j(\yy^n _n),
        \sigma_j(\yy^n _n))
       \dd s \Big|\\
       & =:R_{71}+R_{72}+R_{73}.
\end{align*}
Next, we estimate $R_{71}$, as the other two terms can be estimated similarly. 
In view of Lemma \ref{lm-Du}, we obtain
\begin{align*}
        R_{71}
        & \leq
        C_{\star}
        \sum_{n=0}^{N-N_1-2} 
        \sum_{j=1}^m 
       \ee
        \int_{t_n}^{ t_{n+1} } 
        e ^ {-\lambda(T-s-1) } 
        [
        1+
%        (
        \1_{1 < \gamma \le 2}
        |\yy^n _n|
        +
        \1_{\gamma >2}
        |\yy^n _n|
        ^{\gamma-1}
%        )
        ]\\
        & \qquad \times
        |
        D \sigma_j(\yy^n _n)
        b_\tau
        (
       \yy^n_n
       )
        (s-t_n)
        +
       \RR _{\sigma_j}
      (
       \yy^n,
       \yy^n_n
        )
        |
        \cdot
        |
        \sigma_j(\yy^n _n)
     |
       \dd  s\\
        & \leq
        C_{\star}
        \sum_{n=0}^{N-N_1-2} 
        \sum_{j=1}^m 
       \ee
        \int_{t_n}^{ t_{n+1} } 
        e ^ {-\lambda(T-s-1) } 
        [
        1+
%        (
        \1_{1 < \gamma \le 2}
        |\yy^n _n|
        +
        \1_{\gamma >2}
        |\yy^n _n|
        ^{\gamma-1}
%        )
        ]\\
        & \qquad \times
        |
        D \sigma_j(\yy^n _n)
        b_\tau
     (
       \yy^n_n
      )
        (s-t_n)
    |
        \cdot
        |
        \sigma_j
        (\yy^n _n)
       |
       \dd  s\\
        & \quad +
        C_{\star}
        \sum_{n=0}^{N-N_1-2} 
        \sum_{j=1}^m 
       \ee
        \int_{t_n}^{ t_{n+1} } 
        e ^ {-\lambda(T-s-1) } 
        [
        1+
%        (
        \1_{1 < \gamma \le 2}
        |\yy^n _n|
        +
        \1_{\gamma >2}
        |\yy^n _n|
        ^{\gamma-1}
%        )
        ]\\
        & \qquad \times
        |
       \RR _{\sigma_j}
       (
       \yy^n,
       \yy^n_n
      )
      |
        \cdot
      |
        \sigma_j
        (\yy^n _n)
        |
       \dd  s\\
        &=:R_{711}+R_{712},
\end{align*}
where the  H\"older inequality ensures
\begin{align*}
    R_{711}
    \leq
    C
    \big[
    1+
%   (
    \1_{1 < \gamma \le 2}
    |x_0|_{L_\omega^{4\gamma+1}}^{2\gamma+1}
    +
    \1_{\gamma >2}
    |x_0|_{L_\omega^{5\gamma-1}}^{3\gamma-1}
%   )
    \big] 
    \tau.
\end{align*}
Next, we estimate $R_{712}$ for three cases depending on the ranges of $\gamma$.

(1) In the case $1 < \gamma \leq2$, 
using \eqref{Ds-v}, H\"older inequality,  and Theorem \ref{thm-Y} gives
\begin{align*}
    R_{712}
    & \leq
    C_{\star}
        \sum_{n=0}^{N-N_1-2} 
        \sum_{j=1}^m 
        \int_{t_n}^{ t_{n+1}} 
        e^{-\lambda(T-s-1)} 
        (
        1+
        \sup_{0\leq r \leq N}
        |Y_r|_{L_\omega^{3\gamma+2}}
        )\\
        & \qquad \times
       |
       \RR _{\sigma_j}
        (\yy^n,
       \yy^n_n)
       |
        _{L_\omega^{(3\gamma+2)/2\gamma}}
      |
        \sigma_j(\yy^n _n)
        |
        _{L_\omega^{(3\gamma+2)/(\gamma+1)}}
       \dd  s \\
        & \leq
        C \big (1+
        |x_0|
        ^{(5\gamma+3)/2}
        _{L_\omega^{3\gamma+2}}
        \big) \tau.   
\end{align*}

(2) In the case $2<\gamma\leq 3$,
 we deduce
\begin{align*}
    R_{712}
    & \leq
    C_{\star}
    \sum_{n=0}^{N-N_1-2} 
    \sum_{j=1}^m 
    \int_{t_n}^{ t_{n+1}} 
    e^{-\lambda(T-s-1)} 
    (
    1+
    \sup_{0\leq r \leq N}
    |Y_r|^{\gamma-1}
    _{L_\omega^{4\gamma}}
    )\\
    & \qquad \times
    |
   \RR _{\sigma_j}
    (\yy^n,
   \yy^n_n)
   |
    _{L_\omega^{2}}
    |
    \sigma_j(\yy^n _n)
   |
    _{L_\omega^{4\gamma/(\gamma+1)}
   }
   \dd  s \\
    & \leq
    C (1+
    |x_0|
    ^{(7\gamma-1)/2}
    _{L_\omega^{4\gamma}}
    ) \tau.   
\end{align*}

(3) In the case $\gamma>3$,  we get
\begin{align*}
    R_{712}
    & \leq
    C_{\star}
    \sum_{n=0}^{N-N_1-2} 
    \sum_{j=1}^m 
    \int_{t_n}^{ t_{n+1}} 
    e^{-\lambda(T-s-1)} 
    (
    1+
    \sup_{0\leq r \leq N}
    |Y_r|^{\gamma-1}
    _{L_\omega^{7\gamma-3}}
    )\\
    & \qquad \times
    |
   \RR _{\sigma_j}
    (\yy^n,
   \yy^n_n)
    |
    _{L_\omega^{(7\gamma-3)/(5\gamma-3)}}
    |
    \sigma_j(\yy^n_n)
 |
    _{L_\omega^{4\gamma/(\gamma+1)}
    }
   \dd  s \\
    & \leq
    C  (1+
    |x_0|
    ^{4\gamma-2}
    _{L_\omega^{7\gamma-3}}
    ) \tau.  
\end{align*}
%
%
%
\iffalse
Given the above estimates, we have
\begin{align*}
    T_{712}
    & \leq
    C
    [
    1+
    \1_{1 < \gamma \le 2}
    |x_0|_{L_\omega^{3 \gamma+2}}
    ^{(5\gamma+3)/2}
    +
    \1_{\gamma \in(2,3]}
    |x_0|_{L_\omega^{4\gamma}
    }
    ^{(7\gamma-1)/2}
  +
    \1_{\gamma \in(3,\infty)}
    |x_0|_{L_\omega^{7\gamma-3}}^{ 4\gamma-2}
    ]\tau.    
\end{align*}
\fi
%
%
%
The proofs for $R_{72}$ and $R_{73}$ follow similarly and are therefore omitted.
Then
\begin{align*}
    |R_7 |
    & \leq
    C
    \Big [
    1+
%   (
    \1_{1 < \gamma \le 2}
    |x_0|_{L_\omega^{4\gamma+1}}^{(5\gamma+3)/2}
    +
    \1_{\gamma \in(2,3]}
    |x_0|_{L_\omega^{5\gamma-1}
    }
    ^{(7\gamma-1)/2}
   +
    \1_{\gamma \in(3,\infty)}
    |x_0|_{L_\omega^{7\gamma-3}}^{ 4\gamma-2}
    \Big] 
    \tau.    
\end{align*}
To overcome the possible singularities, we  decompose $|R_8 |$ as follows:
\begin{align*}
        |R_8 |
        & \leq
        \frac{1}{2} 
        \Big [
         \sum_{n=0}^{N-N_1-2}  
          +
           \sum_{n=N-N_1-1}^{n=N-N_1-1}
         +
          \sum_{n=N-N_1}^{N-2}  
               \Big ]
        \sum_{j=1}^m 
       \ee
        \int_{t_n}^{t_{n+1}}
        \int_0^1 
        |
        D^3 u
      (T-s, v_3(\bar{r}))
       \\
        & \qquad
        (
        \sigma_j(\yy^n)
        -
        \sigma_j(\yy^n_n),
        \sigma_j(\yy^n )
        -
        \sigma_j(\yy^n_n),
        \sigma_j(\yy^n_n)
        )
        |
       \dd  \bar{r}
       \dd s\\
       &\quad +
        \frac{1}{2}
        \sum_{j=1}^m 
       \ee
        \int_{T-\tau}^{T}
        |
      (  D^2 
        u(T-s,\yy^{N-1})-
         D^2 
        u(T-s,\yy^{N-1}_{N-1})) \\
        &\qquad
        (
        \sigma_j(\yy^{N-1}))-
        \sigma_j(\yy^{N-1} _{N-1}),
        \sigma_j(\yy^{N-1} _{N-1})
        )
        |
       \dd s\\
       &
       =:R_{81}+R_{82}+R_{83}+R_{84},
\end{align*}
where $v_3(\bar{r})
:=
\yy^n _n
+
\bar{r}(\yy^n (s)-\yy^n_n)$. 
Thanks to Lemmas \ref{lm-Du},  \ref{lm-yy}, and the condition \eqref{s-grow}, we derive
\begin{align*}
        R_{81}
        & \leq
        C_{\star}
        \sum_{n=0}^{N-N_1-2} 
        \sum_{j=1}^m 
       \ee
        \int_{t_n}^{t_{n+1}}
        \int_0^1
        e^{-\lambda(T-s-1)}
        [
        1+
        \1_{1 < \gamma \le 2}
        |v_3(\bar{r})|
        \\
        & \qquad +
        \1_{\gamma >2}
        |v_3(\bar{r})|
        ^{\gamma-1})
        ]
    |
        \sigma_j(\yy^n)
        - \sigma_j(\yy^n_n)
        |^2
     |
        \sigma_j(\yy^n _n)
        |
       \dd  \bar{r}
       \dd  s\\
        & \leq 
        C [
       1+
%   (
    \1_{1 < \gamma \le 2}
    |x_0|_{L_\omega^{4\gamma+1}}^{(7\gamma+1)/2}
    +
    \1_{\gamma >2}
    |x_0|_{L_\omega^{5\gamma-1}}^{(9\gamma-3)/2}
%   )
    ] \tau.
\end{align*}
In a similar way,   $R_{82}$ can be estimated as follows: 
\begin{align*}
    R_{82} 
    \leq
    C
    \Big [
    1+
%   (
    \1_{1 < \gamma \le 2}
    |x_0|_{L_\omega^{4\gamma+1}}
    ^{(7\gamma+1)/2}
    +
    \1_{\gamma >2}
    |x_0|_{L_\omega^{5\gamma-1}}^{(9\gamma-3)/2}
    \Big ]
    \tau(1+ |\ln \tau|).
\end{align*}
Likewise, we obtain
\begin{align*}
    R_{83}
    & \leq
    C
    \sum_{n=N-N_1}^{N-2}  
    \sum_{j=1}^m 
   \ee
    \int_{t_n}^{t_{n+1}}
    \int_0^1
    \frac{1}{T-s}
    [
    1+
    \1_{1 < \gamma \le 2}
    |v_3(\bar{r})|
  \\
    & \qquad   +
    \1_{\gamma >2}    |v_3(\bar{r})|
    ^{\gamma-1})
    ]
|
    \sigma_j(\yy^n)
    - 
    \sigma_j(\yy^n _n)
    |^2
   |
    \sigma_j(\yy^n _n)
   |
   \dd  \bar{r}
   \dd  s\\
    & \leq
    C
    [
    1+
%   (
    \1_{1 < \gamma \le 2}
    |x_0|_{L_\omega^{4\gamma+1}}
    ^{(7\gamma+1)/2}
    +
    \1_{\gamma >2}
    |x_0|_{L_\omega^{5\gamma-1}}^{(9\gamma-3)/2}
    ]
    \tau | \ln \tau|,
\end{align*}
where we used the fact that
\begin{align*}
    \sum_{n=N-N_1}^{N-2} \int_{t_n}^{t_{n+1}}
    \frac{1}{T-s} 
   \dd  s 
    \leq 
    \int_\tau^1 
    \frac{1}{s} 
   \dd  s
    =
    |\ln \tau| .
\end{align*}
We are now in a position to handle the term $R_{84}$ as follows:
\begin{align*}
       R_{84}
        & \leq  
        \frac{1}{2}
        \sum_{j=1}^m 
       \ee
        \int_{T-\tau}^{T}
        |
        D^2 
        u(T-s,\yy^{N-1})
        (
        \sigma_j(\yy^{N-1}))-
        \sigma_j(\yy^{N-1} _{N-1}),
        \sigma_j(\yy^{N-1} _{N-1})
        )
        |
       \dd s\\
        & \quad +
        \frac{1}{2} 
        \sum_{j=1}^m 
       \ee
        \int_{T-\tau}^{T}
        |
        D^2 
        u(T-s,\yy^{N-1}_{N-1})
        (\sigma_j(\yy^{N-1})
        -
        \sigma_j(\yy^{N-1}_{N-1}),
        \sigma_j(\yy^{N-1}_{N-1})
        )
        |
       \dd s \\
        &=:R_{841}+R_{842}.
\end{align*}
For $R_{841}$, one can easily show
\begin{align*}
        R_{841}
        & \leq
        C 
        \sum_{j=1}^m 
       \ee
        \int_{T-\tau}^T 
        \frac{1}
        {\sqrt{T-s}}
        [
        1+
        \1_{1 < \gamma \le 2}
        |\yy^{N-1}|
        +
        \1_{\gamma >2}
        |\yy^{N-1}|^{\gamma-1}
        ]\\
        & \qquad \cdot
     |
        \sigma_j(\yy^{N-1}_{N-1})
        -
        \sigma_j(\yy^{N-1})
    |
        \cdot
       |
        \sigma_j(\yy^{N-1}_{N-1})
        |
       \dd  s\\
%        & \leq 
%        C h^{1 / 2}
%        \int_{T-h}^T 
%        \frac{1}
%        {\sqrt{T-s}} 
%       \dd  s\\
%       & \quad \quad \cdot
%        \left[
%        \1_{1 < \gamma \le 2}
%         (
%         1+|x_0|_{L^{2 (\gamma+1)}(\Omega, \rr^d)}^{3(\gamma+1)/2}
%         )
%          +
%         \1_{\gamma >2}
%        (
%        1+\left|x_0\right|_{L^{5 \gamma}(\Omega, \rr^d)}^{(5\gamma-1)/2})
%        \right] \\
        & \leq 
        C
        [
        1+
        \1_{1 < \gamma \le 2}
        |x_0|
        _{L_\omega^{4\gamma+1}}
        ^{2\gamma+1}
        +
        \1_{\gamma >2}
        |x_0|_{L_\omega^{5\gamma-1}}
        ^{3\gamma-1}
        ]\tau,
\end{align*}
where the fact was used that
\begin{align*}
    \int_{T-\tau}^T 
        \frac{1}
        {\sqrt{T-s}} 
       \dd  s 
        =
        \int_{0}^\tau 
        \frac{1}
        {\sqrt{t}} 
       \dd  t 
        =
        2\sqrt{\tau}.
\end{align*}
%
%%%%%%%%%%%%%%
\iffalse
{\color{blue}
\textbf{Case I:$1<\gamma\leq2$.}
\begin{align*}
    T_{8,4,1}
    & \leq
    C 
   \ee
    \int_{T-h}^T 
    \frac{1}
    {\sqrt{T-s}}
    |\yy  ^{N-1}(s)|_{L^{4\gamma+1}(\Omega,\rr^d)}\\
    & \qquad \times
    |
    \sigma_j(\yy  ^{N-1}(t_{N-1}))
    -
    \sigma_j(\yy  ^{N-1}(s))
    |_{L^{(4\gamma+1)/(3\gamma-1)}(\Omega,\rr^d)}\\
    & \quad \times
    |
    \sigma_j(\yy  ^{N-1}(t_{N-1}))
    |
    _{L^{(4\gamma+1)
    /(\gamma+1)}
    (\Omega,\rr^d)}
   \dd  s\\
    & \leq
    C
    (
    1+
    |x_0|
    _{L^{4\gamma+1}(\Omega, \rr^d)}
    ^{2\gamma+1}
    ).
\end{align*}
\textbf{Case I:$1<\gamma\leq2$.}
\begin{align*}
    T_{8,4,1}
    & \leq
    C 
   \ee
    \int_{T-h}^T 
    \frac{1}
    {\sqrt{T-s}}
    (
    1+
    |\yy  ^{N-1}(s)|^{\gamma-1}_{L^{5\gamma-1}(\Omega,\rr^d)}
    )\\
    & \qquad \times
    |
    \sigma_j(\yy  ^{N-1}(t_{N-1}))
    -
    \sigma_j(\yy  ^{N-1}(s))
    |_{L^{(5\gamma-1)/(3\gamma-1)}(\Omega,\rr^d)}\\
    & \quad \times
    |
    \sigma_j(\yy  ^{N-1}(t_{N-1}))
    |
    _{L^{(5\gamma-1)/(\gamma+1)}
    (\Omega,\rr^d)}
   \dd  s\\
    & \leq
    C
    (
    1+
    |x_0|
    _{L^{5\gamma-1}(\Omega, \rr^d)}
    ^{3\gamma-1}
    ).
\end{align*}
}
\fi
%%%%%%%%%%%%%%
%
The estimate of $R_{842}$ follows similarly, and we omit the details here.
As a consequence of the above estimates, we obtain
\begin{align*}
    |R_8|
    & \leq
    C
    [
    1+
%   (
    \1_{1 < \gamma \le 2}
    |x_0|_{L_\omega^{4\gamma+1}}
    ^{(7\gamma+1)/2}
    +
    \1_{\gamma >2}
    |x_0|_{L_\omega^{5\gamma-1}}^{(9\gamma-3)/2}
    ]
    \tau|\ln \tau|.
\end{align*}
Then we have
\begin{align*}
    |J_{222}|
    & \leq
    C
    [
    1+
    \1_{1 < \gamma \le 2}
    |x_0|_{L_\omega^{4 \gamma+1}}^{(7\gamma+1)/2}
    +
    \1_{2 < \gamma \le 3}
    |x_0|_{L_\omega^{5 \gamma}}^{(7\gamma-1)/2} +
    \1_{\gamma \in(3,\infty)}
    |x_0|_{L_\omega^{7\gamma-3}}^{4\gamma-2}
    ]
    \tau |\ln \tau|,
\end{align*}
and the desired assertion \eqref{J22-est} follows, by taking these estimates into account.
\end{proof}

We are ready to provide the uniform weak error estimate for \eqref{MEM}.

\begin{thm}
\label{thm-weak}
    Let Assumptions \ref{A1}--\ref{A5} hold. 
    %Let $\{X_t^{x_0}\right\}_{t \geq 0}$ and $\left\{Y_n^{x_0}\right\}_{n \geq 0}$ be the solutions of Eq. \eqref{SDE} and the modified Euler algorithms \eqref{yy} with the same initial state $X_0^{x_0}=Y_0^{x_0}=x_0$, respectively. 
    %Also, let $h \in\left(0, \min \left\{1 / 2 a_1, 2 a_1 /\left(a_1+2 C_f^2\right), 1\right\}\right)$, where $a_1$ and $C_f$ are given in Assumption 2.4 and Lemma 3.3, respectively, be the uniform timestep. 
For any $\phi \in \mathrm{\rm Lip}(\rr^d)$  and $x_0 \in L^{l_1 \vee l_3 \vee l_5}(\Omega; \rr^d)$, there exists a constant $C$ such that for any $\tau \in (0, 1)$ and $N \in \nn$,
    \begin{align}
    \label{weak}
       &  |
       \ee \phi(Y_N^{x_0})-
       \ee \phi(X_{N\tau}^{x_0})| \nonumber\\
        & \leq 
        C 
      \Big [
        1+
        \1_{1 < \gamma \le 2}
        |x_0|
        _{L_\omega^{l_1}}
        ^{l_2}
        +
        \1_{2 < \gamma \le 3}
        |x_0|_{L_\omega^{l_3}}^{l_4} +
        \1_{\gamma >3}
        |x_0|
        _{L_\omega^{l_5}}^{l_6} \Big]
        \tau|\ln \tau|.
    \end{align}
\end{thm}
\begin{proof}
Owing to the telescoping argument, \eqref{yy} and \eqref{def-u}, the weak error can be decomposed as follows with $T=N\tau$:
\begin{align*}
        |
       \ee
        [
        \phi
       (
        Y^{x_0}_N
        )
        ]
        -
       \ee
        [
        \phi
        (X^{x_0}_T)
       ]
        |
        & =
        \Big |
        \sum_{n=0}^{N-1} 
       \ee
        [
        u
        (
        T-t_{n+1},Y_{n+1}
       )
        ]
        -
       \ee
       [
        u
      (
        T-t_n, Y_n
     )
       ]
        \Big |
        \\
        & \leq 
        \Big |
        \sum_{n=0}^{N-1} 
       \ee
        [
        u
       (
        T-t_n,\yy^n_n
       )
        ]
        -
       \ee
      [
        u
       (
        T-t_n, Y_n
      )
       ]
      \Big |
        \\
        & \quad + 
        \Big |
        \sum_{n=0}^{N-1} 
       \ee
      [
        u
        (
        T-t_{n+1},\yy^n_{n+1}
       )
       ]
        -
       \ee
        [
        u
       (
        T-t_n,\yy^n_n
        )
       ]
      \Big |
        \\
        &=: |J_1|+|J_2|.
\end{align*}
For  the term $J_2$, we first use the It\^o formula and the corresponding Kolmogorov equation \eqref{Komgv} to deduce
\begin{align*}
        &\quad 
        u( T-t_{n+1},\yy^n_{n+1})-u(  T-t_n,\yy^n_n ) \\
        %& = -  \int_{t_n}^{t_{n+1}} \partial_s u(T-s,\yy^n)  \dd  s+
        %\int_{t_n}^{t_{n+1}} D u (T-s,\yy^n)  b_\tau(\yy^n_n)  \dd  s \\
       % & \quad +\sum_{j=1}^m \int_{t_n}^{t_{n+1}} D u(T-s,\yy^n ) \sigma_{j,\tau}(\yy^n_n)\dd  W_s\\
      %  & \quad + \frac{1}{2} \sum_{j=1}^m \int_{t_n}^{t_{n+1}} D^2 u(T-s,\yy^n)(\sigma_{j,\tau}(\yy^n_n), \sigma_{j,\tau}(\yy^n_n) ) \dd  s \\
        & = 
        \int_{t_n}^{t_{n+1}} D u(T-s,\yy^n)(b_\tau (
       \yy^n_n
        )
        -
        b
        (
       \yy^n
        )
        ) 
       \dd  s\\
        & \quad +
        \sum_{j=1}^m 
        \int_{t_n}^{t_{n+1}} 
        D u
        (
        T-s,\yy^n
        ) 
        \sigma_{j,\tau}
        (
       \yy^n_n
        ) 
       \dd  W_s\\
        & \quad +
        \frac{1}{2}
        \sum_{j=1}^m \int_{t_n}^{t_{n+1}} 
        D^2 u
        (
        T-s,\yy^n
       )
        (
        \sigma_{j,\tau}
        (
       \yy^n_n
        ), 
        \sigma_{j,\tau}
       (
       \yy^n_n
       )
        )\\
        & \qquad \quad -
        D^2 u
       (
        T-s,\yy^n
        )
        (
        \sigma_{j}
       (
       \yy^n
       ), 
        \sigma_{j}
       (
       \yy^n
       )
        )
       \dd s.
\end{align*}
Taking the conditional expectation argument, we then obtain
\begin{align*}
        J_2 
        & =
        \sum_{n=0}^{N-1} 
       \ee
        \int_{t_n}^{t_{n+1}}
        D u
       (
        T-s,\yy^n
       )
       (
        b_\tau
       (
       \yy^n_n
        )
        -
        b
        (
       \yy^n
       )
        ) 
       \dd  s \\
        & \quad +
        \frac{1}{2}
         \sum_{n=0}^{N-1} 
        \sum_{j=1}^m
       \ee
        \int_{t_n}^{t_{n+1}}
        D^2 u
       (
        T-s,\yy^n
        )
        (
        \sigma_{j,\tau}
       (
       \yy^n_n
        ), 
        \sigma_{j,\tau}
        (
       \yy^n_n
        )
        )  \\
        & 
        \qquad  -
        D^2 u
       (
        T-s,\yy^n
        )
        (
        \sigma_{j}
       (
       \yy^n
        ), 
        \sigma_{j}
       (
       \yy^n
       )
        )
       \dd  s
        \\
        &=:J_{21}+J_{22}.
\end{align*}
According to Lemmas \ref{lm-J1}, \ref{lm-J21}, and \ref{lm-J22}, we arrive at the desired assertion \eqref{weak}.
\end{proof}

{\bf Proof of Theorem \ref{thm-W1}}:
Theorem \ref{thm-W1} follows from \eqref{ex-decay-x} and Theorem \ref{thm-weak}.
\qed

\bibliographystyle{abbrv}
\bibliography{bib}

\begin{thebibliography}{10}

\bibitem{BIK16}
W.-J. Beyn, E.~Isaak, and R.~Kruse.
\newblock Stochastic {C}-stability and {B}-consistency of explicit and implicit
  {E}uler-type schemes.
\newblock {\em J. Sci. Comput.}, 67(3):955--987, 2016.

\bibitem{bre14}
C.-E. Br\'ehier.
\newblock Approximation of the invariant measure with an {E}uler scheme for
  stochastic {PDE}s driven by space-time white noise.
\newblock {\em Potential Anal.}, 40(1):1--40, 2014.

\bibitem{brehier2023approximation}
C.-E. Brehier.
\newblock Approximation of the invariant distribution for a class of ergodic
  {SDE}s with one-sided {L}ipschitz continuous drift coefficient using an
  explicit tamed {E}uler scheme.
\newblock {\em ESAIM Probab. Stat.}, 27:841--866, 2023.

\bibitem{cer01}
S.~Cerrai.
\newblock {\em Second order {PDE}'s in finite and infinite dimension}, volume
  1762 of {\em Lecture Notes in Mathematics}.
\newblock Springer-Verlag, Berlin, 2001.
\newblock A probabilistic approach.

\bibitem{EKL94}
K.~D. Elworthy and X.-M. Li.
\newblock Formulae for the derivatives of heat semigroups.
\newblock {\em J. Funct. Anal.}, 125(1):252--286, 1994.

\bibitem{fang2020adaptive}
W.~Fang and M.~B. Giles.
\newblock Adaptive {E}uler-{M}aruyama method for {SDE}s with nonglobally
  {L}ipschitz drift.
\newblock {\em Ann. Appl. Probab.}, 30(2):526--560, 2020.

\bibitem{HJK11}
M.~Hutzenthaler, A.~Jentzen, and P.~E. Kloeden.
\newblock Strong and weak divergence in finite time of {E}uler's method for
  stochastic differential equations with non-globally {L}ipschitz continuous
  coefficients.
\newblock {\em Proc. R. Soc. Lond. Ser. A Math. Phys. Eng. Sci.},
  467(2130):1563--1576, 2011.

\bibitem{LWW24}
L.~Li, M.~Wang, and Y.~Wang.
\newblock {Error estimates of the Euler's method for stochastic differential
  equations with multiplicative noise via relative entropy}.
\newblock {\em arXiv:2409.04991}, 2024.

\bibitem{LWX23}
X.~Li, F.~Wang, and L.~Xu.
\newblock {Unadjusted Langevin algorithms for SDEs with H\"older drift}.
\newblock {\em Sci. China Math.}, 2025.

\bibitem{liu2023backward}
W.~Liu, X.~Mao, and Y.~Wu.
\newblock The backward {E}uler-{M}aruyama method for invariant measures of
  stochastic differential equations with super-linear coefficients.
\newblock {\em Appl. Numer. Math.}, 184:137--150, 2023.

\bibitem{LW24}
Z.~Liu and X.~Wu.
\newblock {Geometric ergodicity and strong error estimates for tamed schemes of
  super-linear {SODE}s}.
\newblock {\em arXiv:2411.06049}, 2024.

\bibitem{LS25}
I.~Lytras and S.~Sabanis.
\newblock Taming under isoperimetry.
\newblock {\em Stochastic Process. Appl.}, 188:Paper No. 104684, 2025.

\bibitem{NNZ25}
A.~Neufeld, M.~Ng, and Y.~Zhang.
\newblock Non-asymptotic convergence bounds for modified tamed unadjusted
  {L}angevin algorithm in non-convex setting.
\newblock {\em J. Math. Anal. Appl.}, 543(1):Paper No. 128892, 53, 2025.

\bibitem{PP23}
G.~Pag\`es and F.~Panloup.
\newblock Unadjusted {L}angevin algorithm with multiplicative noise: total
  variation and {W}asserstein bounds.
\newblock {\em Ann. Appl. Probab.}, 33(1):726--779, 2023.

\bibitem{pang2024antithetic}
C.~Pang and X.~Wang.
\newblock Antithetic multilevel {M}onte {C}arlo method for approximations of
  {SDE}s with non-globally {L}ipschitz continuous coefficients.
\newblock {\em Stochastic Process. Appl.}, 178:Paper No. 104467, 30, 2024.

\bibitem{PWW24}
C.~Pang, X.~Wang, and Y.~Wu.
\newblock Linear implicit approximations of invariant measures of semi-linear
  {SDE}s with non-globally {L}ipschitz coefficients.
\newblock {\em J. Complexity}, 83:Paper No. 101842, 45, 2024.

\bibitem{PWW25}
C.~Pang, X.~Wang, and Y.~Wu.
\newblock Projected {L}angevin {M}onte {C}arlo algorithms in non-convex and
  super-linear setting.
\newblock {\em J. Comput. Phys.}, 526:Paper No. 113754, 33, 2025.

\bibitem{san15}
F.~Santambrogio.
\newblock {\em Optimal transport for applied mathematicians}, volume~87 of {\em
  Progress in Nonlinear Differential Equations and their Applications}.
\newblock Birkh\"auser/Springer, Cham, 2015.
\newblock Calculus of variations, PDEs, and modeling.

\bibitem{wang20}
F.~Wang.
\newblock Exponential contraction in {W}asserstein distances for diffusion
  semigroups with negative curvature.
\newblock {\em Potential Anal.}, 53(3):1123--1144, 2020.

\bibitem{WW24}
X.~Wu and X.~Wang.
\newblock {Strong convergence rates for long-time approximations of SDEs with
  non-globally Lipschitz continuous coefficients}.
\newblock {\em arXiv:2406.10582}, 2024.

\end{thebibliography}

\end{document}